\documentclass[10pt]{amsart}

\usepackage{geometry}
\usepackage{setspace}
\usepackage{enumerate}

\usepackage{amsmath,amssymb,graphicx,amscd}
\usepackage{mathtools}
\usepackage{color}
\usepackage{xcolor}
\usepackage{float}
\usepackage[numbers]{natbib}
\usepackage[all]{xy}
\usepackage{amsthm}
\usepackage{tikz}
\usepackage{fancyhdr}
\usepackage{hyperref}
\hypersetup{colorlinks=true,citecolor=blue,linkcolor=cyan!100,urlcolor=black,filecolor=black}

\numberwithin{equation}{section}
\allowdisplaybreaks
\numberwithin{equation}{section}

\newtheorem{theorem}{Theorem}[section]
\newtheorem{proposition}[theorem]{Proposition}
\newtheorem{corollary}[theorem]{Corollary}
\newtheorem{lemma}[theorem]{Lemma}
\newtheorem{definition}[theorem]{Definition}

\theoremstyle{definition}
\newtheorem{remark}[theorem]{Remark}
\newtheorem{example}[theorem]{Example}

\newcommand{\hooklongrightarrow}{\lhook\joinrel\longrightarrow}

\author{Anderson Vera
}
\address{Institut de Recherche Math\'ematique Avanc\'ee,  Universit\'e de Strasbourg, 7 rue Ren\'e Descartes,  67084 Strasbourg, France}
\email{vera@math.unistra.fr}

\date{\today}

\title{Johnson-Levine Homomorphisms and the tree reduction of the LMO functor}
\begin{document}

\begin{abstract}

Let $\mathcal{M}$ denote the mapping class group of $\Sigma$, a compact connected  oriented surface  with one boundary component. The action of $\mathcal{M}$ on the nilpotent quotients of $\pi_1(\Sigma)$ allows to define the so-called  Johnson filtration and  the Johnson homomorphisms. J. Levine introduced a new filtration of $\mathcal{M}$, called the \emph{Lagrangian filtration}. He also introduced  a version of the Johnson homomorphisms for this new filtration. The first term of the Lagrangian filtration is the \emph{Lagrangian mapping class group}, whose definition involves a handlebody bounded by $\Sigma$, and which contains the Torelli group. These constructions extend in a natural way to the monoid of homology cobordisms. Besides, D. Cheptea, K. Habiro and G. Massuyeau  constructed a functorial extension of the LMO invariant, called the LMO functor, which takes values in a category of diagrams. In this paper we give a topological interpretation of the \emph{upper part} of the tree reduction of the LMO functor in terms of the homomorphisms defined by J. Levine for the Lagrangian mapping class group. We also compare the Johnson filtration with the filtration introduced by J. Levine.
\end{abstract}

\maketitle
\tableofcontents

\section{Introduction}

Let $\Sigma$ be a compact connected  oriented surface  with one boundary component and let $\mathcal{M}$ denote the mapping class group of $\Sigma$.  The interaction between the study of $3$-manifolds and that of the mapping class group is well known. In some sense, the algebraic structure of $\mathcal{M}$ and of its subgroups  is reflected in the topology of $3$-manifolds. For instance, the subgroup of homeomorphisms acting trivially in homology, known as the \emph{Torelli group} and denoted by $\mathcal{I}$, is tied to homology $3$-spheres. In this direction, D. Johnson \cite{MR718141} and S. Morita \cite{MR1224104} studied the  mapping class group by using its action on the nilpotent quotients of the fundamental group of $\Sigma$. This action allows to define the Johnson  filtration of $\mathcal{M}$; the $k$-th term $J_k\mathcal{M}$ of this filtration consists of the elements in $\mathcal{M}$ acting trivially on the $k$-th nilpotent quotient of  $\pi_1(\Sigma)$. On the Johnson filtration it is possible to define  the Johnson homomorphisms which play an important role in the structure of the Torelli group. For instance, the first Johnson homomorphism appears in the computation of the abelianization of $\mathcal{I}$ \cite{MR793179}. S. Morita also  discovered the strong relation between the structure of the Torelli group and some properties of the Casson invariant of homology $3$-spheres \cite{MR1014464,MR1133875,MR1617632}. The Johnson homomorphisms take values in a Lie subalgebra of the derivation Lie algebra of a free Lie algebra constructed from the first homology group of $\Sigma$; this Lie subalgebra admits a diagrammatic description in terms of  \emph{tree-like Jacobi diagrams}.

The Johnson filtration and the Johnson homomorphisms  generalize in a natural way to the monoid of \emph{homology cobordisms} $\mathcal{C}$ of $\Sigma$, that is, homeomorphism classes of pairs $(M,m)$, where $M$ is a compact oriented $3$-manifold and $m:\partial(\Sigma\times[-1,1])\rightarrow\partial M$ is an orientation-preserving homeomorphism such that the \emph{top} and \emph{bottom} restrictions of $m$ induce  isomorphisms in homology \cite{MR2131016}. In particular, the mapping class group of $\Sigma$  embeds into the monoid of homology cobordisms by associating  to each $h\in\mathcal{M}$ the cobordism $(\Sigma\times\left[-1,1\right], m^h)$ where $m^h$  is the orientation-preserving homeomorphism defined on the top surface $\Sigma\times\{1\}$  by $h$ and the identity elsewhere. Under this embedding,  the Torelli group is mapped into the monoid of homology cobordisms $(M,m)$ such that the top and bottom restrictions of $m$ induce the \emph{same} isomorphisms in homology. This class of cobordisms is denoted by $\mathcal{IC}$ and they are called \emph{homology cylinders}.

On the other hand, T. Le, J. Murakami and T. Ohtsuki  defined in \cite{MR1604883} a universal finite type invariant for homology $3$-spheres called the   \emph{LMO invariant}. This invariant was extended by D. Cheptea, K. Habiro and G. Massuyeau in \cite{MR2403806} to a functor $\widetilde{Z}:\mathcal{LC}ob_q\rightarrow {}^{ts}\!\!\mathcal{A}$, called the  \emph{LMO functor}, from a category of cobordisms (with a homological condition)  between bordered surfaces to a category of Jacobi diagrams. In particular, the monoid of homology cylinders $\mathcal{IC}$ is a subset of morphisms in $\mathcal{LC}ob_q$. The construction of the LMO functor is sophisticated: it uses the \emph{Kontsevich integral}, which requires the choice of a \emph{Drinfeld associator}, and it also uses several combinatorial operations in the space of Jacobi diagrams. In consequence, it is not clear which topological  information is encoded by the LMO functor.

In \cite{MR1783857}, N. Habegger and G. Masbaum gave a topological interpretation of the  tree reduction of the Kontsevich integral in terms of Milnor invariants. Following the same spirit,  D. Cheptea, K. Habiro and G. Massuyeau gave in \cite{MR2403806} a topological interpretation of the leading term of the  tree reduction of the LMO functor in terms of the first non-vanishing Johnson homomorphism. This was improved by G. Massuyeau in \cite{MR2903772}, where he gave an interpretation of the full tree reduction of the LMO functor on $\mathcal{IC}$. 

In \cite{MR1823501,MR2265877}, J. Levine introduced a different filtration of the mapping class group as follows. Let $V$ be a handlebody of genus $g$ and fix a disk $D$ on the boundary of $V$ so that $\partial V=\Sigma\cup D$, where $D$ and $\Sigma$ are glued along their boundaries. Denote by $\iota$ the inclusion of $\Sigma$ into $\partial V\subseteq V$. Let us denote by $A$ and $\mathbb{A}$ the subgroups $\text{ker}(H_1(\Sigma)\stackrel{\iota_*}{\longrightarrow} H_1(V))$ and $\text{ker}(\pi_1(\Sigma)\stackrel{\iota_{\#}}{\longrightarrow}\pi_1(V))$, respectively. The \emph{Lagrangian mapping class group} of $\Sigma$, denoted by $\mathcal{L}$, consists of the elements in $\mathcal{M}$ preserving the subgroup $A$. The \emph{strongly Lagrangian mapping class group} of $\Sigma$, denoted by $\mathcal{IL}$,  consists of the elements in $\mathcal{L}$ which are the identity on $A$. The $k$-th term $J_k^L\mathcal{M}$ of the \emph{Lagrangian filtration} of $\mathcal{M}$, which we shall call here the  \emph{Johnson-Levine filtration}, consists of the elements $h$ in $\mathcal{IL}$ such that $\iota_{\#}h_{\#}(\mathbb{A})$ is contained in the $(k+1)$-st term of the lower central series of $\pi_1(V)$. 

 J. Levine also  defined a version of the Johnson homomorphisms for this filtration, that we shall call here the  \emph{Johnson-Levine homomorphisms}, which take values in an abelian group  that can be described in terms of $H_1(V)$. This abelian group also admits a diagrammatic description in terms of tree-like Jacobi diagrams. One of J. Levine's  main motivations was to understand the  relation between  the Johnson-Levine homomorphisms  and finite type invariants of homology spheres.   The first Johnson-Levine homomorphism  comes up in the computation of the abelianization of $\mathcal{IL}$, found by T. Sakasai in \cite{MR2916276}. It also appears in the work of N. Broaddus, B. Farb and A. Putman  \cite{MR2803852} to compute the distortion of $\mathcal{IL}$ as a subgroup of $\mathcal{M}$. 

The Johnson-Levine filtration and the Johnson-Levine homomorphisms  generalize in a natural way to the monoid of homology cobordisms. Thus, it is natural to wonder about the relation of these homomorphisms with the LMO functor. The aim of this paper is to make explicit this relation.  The main result is a topological interpretation of the leading term in the  \emph{upper part} of the tree reduction of the LMO functor in terms of the first non-vanishing Johnson-Levine homomorphism. This sheds some new light on the topological information encoded by the LMO functor. One key point in the proof of this result is to compare the Johnson filtration and the Johnson-Levine filtration. This comparison was already carried out by J. Levine in degrees $1$ and $2$ for the mapping class group in \cite{MR2265877}. In this direction, a second main result of this paper is a comparison of the two filtrations in all degrees for homology cobordisms up to some surgery equivalence relations. These equivalence relations were introduced independently by M. Goussarov in \cite{MR1793618,MR1715131} and by K. Habiro in \cite{MR1735632} in connection with the theory of finite type invariants.

The organization of the paper is as follows. In Section \ref{seccion2} we review the definitions of the Johnson filtration and Johnson homomorphisms in the mapping class group case, as well as in the case of  homology cobordisms. We also  explain the bottom-top tangle presentation of homology cobordisms, which is a way to present homology cobordisms by using a kind of knotted objects. Finally, in this section,  we review the   Milnor-Johnson correspondence which relates the Milnor invariants with the Johnson homomorphisms.  Section \ref{seccion3} deals with the Johnson-Levine filtration and Johnson-Levine homomorphisms in the mapping class group case, as well as in the case of  homology cobordisms.  Section \ref{seccion4} provides a detailed exposition of important properties of the Johnson-Levine homomorphisms, and a comparison of the Johnson filtration with the Johnson-Levine filtration. Finally, Section \ref{seccion5} is devoted to the topological interpretation of the upper part of the tree reduction of the LMO functor.

\medskip

\noindent\textbf{Notation.} For a group $G$, the \emph{lower central series} is the descending chain of subgroups $\{\Gamma_k G\}_{k\geq1}$ defined by $\Gamma_1 G:=G$ and $\Gamma_{k+1}G:=[G,\Gamma_{k}G]$. If $x\in G$  we denote the \emph{nilpotent} class of $x$ in $G/\Gamma_k G$ interchangeably by $\{x\}_k$ or $x\Gamma_k G$. If $f:(X,x)\rightarrow (Y,y)$ is a continuous map between two pointed topological spaces  $(X,x)$ and $(Y,y)$, we denote by $f_{\#}:\pi_1(X,x)\rightarrow\pi_1(Y,y)$ and $f_{*}:H_1(X;\mathbb{Z})\rightarrow H_1(Y;\mathbb{Z})$ the induced maps in homotopy and homology, respectively. Finally, when we draw framed knotted objects we use the blackboard framing convention. 

\medskip 

\noindent\textbf{Acknowledgements.} I am deeply grateful to my advisor Gw\'ena\"{e}l Massuyeau for his  encouragement, helpful advice and careful reading.


\section{Johnson homomorphisms}\label{seccion2}

For every non-negative integer $g$ denote by $\Sigma$ (or  by $\Sigma_{g,1}$ if there is  ambiguity) a compact connected oriented surface of genus $g$ with  one boundary component.  Let us fix a base point $*\in\partial\Sigma$ and set $\pi:=\pi_1(\Sigma,*)$ and $H:=H_1(\Sigma;\mathbb{Z})$.
\subsection{Mapping class group}

Denote by $\mathcal{M}$ (or  by $\mathcal{M}_{g,1}$ if there is ambiguity) the \emph{mapping class group} of $\Sigma$, that is, the group of isotopy classes of orientation-preserving diffeomorphisms of $\Sigma$ fixing $\partial \Sigma$ point-wise. The isotopy class of $h$ in $\mathcal{M}$ is still denoted by $h$.
The \emph{Dehn-Nielsen-Baer representation} is the injective group homomorphism
$$\rho:\mathcal{M}\longrightarrow \text{Aut}(\pi),$$
\noindent that maps the isotopy class $h\in \mathcal{M}$ to the induced map in homotopy $h_\#\in\text{Aut}(\pi)$. 

Consider the lower central series $\{\Gamma_k\pi\}_{k\geq1}$ of $\pi$. The \emph{nilpotent version}  of the Dehn-Nielsen-Baer representation,  $\rho_{k}:\mathcal{M}\rightarrow \text{Aut}(\pi/\Gamma_{k+1}\pi)$, is  defined as the composition
\begin{equation}\label{ecuacion2.1}
\mathcal{M}\stackrel{\rho}{\longrightarrow}\text{Aut}(\pi)\longrightarrow \text{Aut}(\pi/\Gamma_{k+1}\pi).
\end{equation}

The \emph{Johnson filtration} is the descending chain of subgroups $\{J_k\mathcal{M}\}_{k\geq1}$ of $\mathcal{M}$   where $J_k\mathcal{M}$ is the kernel of $\rho_k$. In particular, $J_1\mathcal{M}$ is the set of elements in $\mathcal{M}$ acting trivially in homology. This subgroup is denoted by $\mathcal{I}$ (or $\mathcal{I}_{g,1}$) and it is called the \emph{Torelli group}.

Associated to the Johnson filtration, there is a family of group homomorphisms called the \emph{Johnson homomorphisms}. These homomorphisms are of great  importance in the study of the structure of the mapping class group and its subgroups. They were introduced by D. Johnson in \cite{MR579103,MR718141} and extensively studied by S. Morita in \cite{MR1133875,MR1224104}. We refer to \cite{MR3380338} for a survey on this subject.

For every positive integer $k$, the $k$-th \emph{Johnson homomorphism}
\begin{equation}\label{ecuacion2.2}
\tau_k: J_k\mathcal{M}\longrightarrow \text{Hom}(H,\Gamma_{k+1}\pi/\Gamma_{k+2}\pi)\cong H^*\otimes \Gamma_{k+1}\pi/\Gamma_{k+2}\pi \cong H\otimes \mathfrak{L}_{k+1}(H),
\end{equation}
\noindent is defined by sending the isotopy class  $h\in J_k\mathcal{M}$ to the map
$$\{x\}_2\longmapsto \rho_{k+1}(h)(\{x\})\{x\}_{k+1}^{-1}\in\frac{\Gamma_{k+1}\pi}{\Gamma_{k+2}\pi},$$

\noindent for all $x\in\pi$. The second isomorphism in  (\ref{ecuacion2.2}) is given by  the identification $H\stackrel{\sim}{\longrightarrow} H^*$ that maps $x$ to $\omega(x, \cdot)$ where  $\omega:H\otimes H\rightarrow\mathbb{Z}$ is the \emph{intersection form},  together with the identification of $\Gamma_{k+1}\pi/\Gamma_{k+2}\pi$ with the term of degree $k+1$ in the \emph{free Lie algebra}  
$$\mathfrak{L}(H)=\bigoplus_{k\geq 1}\mathfrak{L}_{k}(H)$$
generated by the $\mathbb{Z}$-module $H$.
Moreover, S. Morita proved in \cite[Corollary 3.2]{MR1224104} that the $k$-th Johnson homomorphism takes values in the kernel $D_k(H)$ of the Lie bracket  $\left[\ ,\ \right]: H\otimes\mathfrak{L}_{k+1}(H)\rightarrow\mathfrak{L}_{k+2}(H)$.

\subsection{Homology cobordisms and bottom-top tangles}\label{subsection2.2}

In this subsection we recall from  \cite{MR2403806}   the definition of the monoid of homology cobordisms and their presentation by bottom-top tangles, that is, a presentation  by a special kind of  knotted objects. The bottom-top tangle presentation is also used in the definition of the LMO functor as we will see in Section \ref{seccion5}. 

The notion of homology cobordism was introduced independently by M.  Goussarov in \cite{MR1715131} and by K. Habiro in \cite{MR1735632} in connection with the theory of finite type invariants. A \emph{homology cobordism} of $\Sigma$ is the equivalence class of a pair $M=(M,m)$, where $M$ is a compact connected oriented 3-manifold and $m:\partial(\Sigma\times[-1,1])\rightarrow\partial M$ is an orientation-preserving homeomorphism, such that the \emph{bottom} and \emph{top} inclusions
$m_{\pm}(\cdot):=m(\cdot,\pm1):\Sigma\rightarrow M$ induce isomorphisms in homology. Two pairs $(M,m)$ and $(M',m')$ are \emph{equivalent} if there exists an orientation-preserving homeomorphism $\varphi:M\rightarrow M'$ such that $\varphi\circ m=m'$.

The \emph{composition} $(M,m)\circ (M',m')$ of two homology cobordisms $(M,m)$ and $(M',m')$ of $\Sigma$ is the equivalence class of the pair $(\widetilde{M},m_-\cup m'_+)$, where $\widetilde{M}$ is obtained by gluing the two 3-manifolds $M$ and $M'$ by using the map $m_+\circ(m'_-)^{-1}$. This composition is associative and has as identity element the equivalence class of the trivial cobordism $( \Sigma\times[-1,1], \text{Id})$.  Let us denote by $\mathcal{C}$ (or by $\mathcal{C}_{g,1}$ if there is ambiguity) the \emph{monoid of homology cobordisms} of $\Sigma$. 

\begin{example}\label{ejemplo1} The mapping class group $\mathcal{M}$ can be embedded into $\mathcal{C}$ by associating to any $h\in\mathcal{M}$ the equivalence class of the pair $(\Sigma\times[-1,1], m^h)$, where $m^h:\partial(\Sigma\times[-1,1])\rightarrow\partial (\Sigma\times[-1,1])$ is the orientation-preserving homeomorphism defined by $m^h(x,1)=(h(x),1)$ and  $m^h(x,t)=(x,t)$ for $t\not=1$. The submonoid obtained in this way is precisely the group of  invertible elements of $\mathcal{C}$, see \cite[Proposition 2.4]{MR2952770}. 
\end{example}

Let us now turn to the definition of bottom-top tangles. Consider the square $[-1,1]^2$. For all $g\geq 1$, fix $g$ pairs of different points $(p_1,q_1)$,  $\ldots$, $(p_g, q_g)$ in $[-1,1]^2$ distributed uniformly along the horizontal axis $\{(x,0)\ |\ x\in[-1,1]\}$, see Figure \ref{figura4.5}$(a)$. A \emph{bottom-top tangle} of type $(g,g)$ is an equivalence class of  pairs $(B,\gamma)$, where $B=(B,b)$ consists of a compact connected oriented $3$-manifold $B$ and an orientation-preserving homeomorphism $b:\partial([-1,1]^3)\rightarrow \partial B$; and $\gamma=(\gamma^+,\gamma^-)$ is a framed oriented tangle with $g$ top components $\gamma_1^+,\ldots,\gamma_g^+$ and $g$ bottom components $\gamma_1^-,\ldots,\gamma_g^-$ such that
\begin{itemize}
\item each $\gamma_j^+$ runs from $p_j\times 1$ to $q_j\times1$,
\item each $\gamma_j^-$ runs from $q_j\times (-1)$ to $p_j\times(-1)$.
\end{itemize}

\noindent Two such pairs $(B,\gamma)$ and $(B',\gamma')$ are \emph{equivalent} if there is an orientation-preserving homeomorphism $\varphi:B\rightarrow B'$ such that $\varphi\circ b=b'$ and $\varphi(\gamma)=\gamma'$. See Figure \ref{figura4.5}$(b)$ for an example.
\begin{figure}[ht!]
										\centering
                        \includegraphics[scale=0.57]{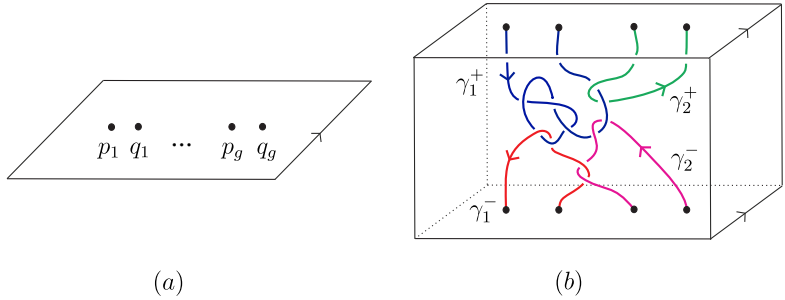}
												\caption{($b$)  bottom-top tangle of type $(2,2)$ in $[-1,1]^{3}$.}
												\label{figura4.5}
\end{figure}

Let $(M,m)$ be a homology cobordism of  $\Sigma_{g,1}$. We associate a bottom-top tangle of type $(g,g)$ to $(M,m)$ as follows. Let us   fix a system of meridians and parallels $\{\alpha_1,\ldots,\alpha_g,\beta_1,\ldots,\beta_g\}$ of $\Sigma_{g,1}$ as  in  Figure \ref{figura3.0}. 

\begin{figure}[ht!] 
										\centering
                        \includegraphics[scale=0.65]{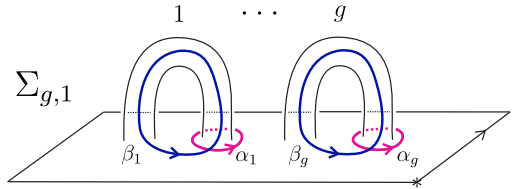}
												\caption{System of meridians and parallels.}
												\label{figura3.0}
\end{figure}

\noindent Then  attach $g$  $2$-handles  on the bottom surface of $M$ by sending the cores of the $2$-handles to the curves $m_-(\alpha_i)$. In the same way,  attach $g$ $2$-handles on the top surface of $M$ by sending the cores to the curves $m_+(\beta_i)$. This way we obtain a compact connected oriented $3$-manifold $B$ and an orientation-preserving homeomorphism $b:\partial([-1,1]^3)\rightarrow\partial B$. The pair $B=(B,b)$ together with the cocores of the $2$-handles, determine a bottom-top tangle $(B,\gamma)$ of type $(g,g)$. We call $(B,\gamma)$ the \emph{bottom-top tangle presentation} of $(M,m)$. See Figure \ref{figura4.7} for an example.

\begin{figure}[ht!] 
										\centering
                        \includegraphics[scale=0.7]{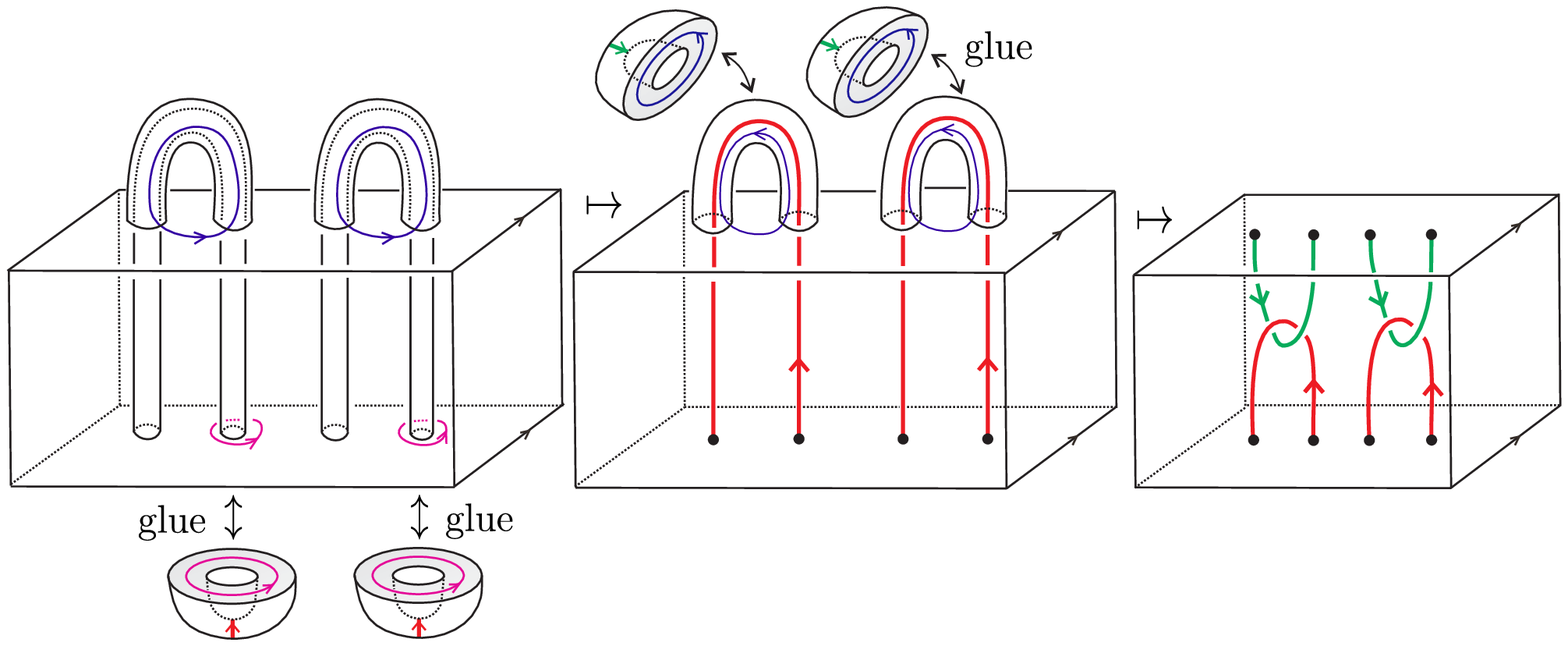}
												\caption{From homology cobordisms to bottom-top tangles.}
												\label{figura4.7}
\end{figure}

We emphasize that the bottom-top tangle presentation of homology cobordisms depends on the choice of a system of meridians and parallels  of $\Sigma$. From now on, when we say ``the bottom-top tangle presentation'' of a homology cobordism we mean the bottom-top tangle presentation associated to the choice of meridians and parallels of $\Sigma$ as  in  Figure \ref{figura3.0}.

We are mainly interested in bottom-top tangles in homology cubes. A \emph{homology cube} is a homology cobordism of  $\Sigma_{0,1}$. In particular, if $(B,b)$ is such a cobordism we have $H_*(B;\mathbb{Z})\cong H_*([-1,1]^3;\mathbb{Z})$.

\begin{definition}\label{linkingbt}
Let $(B,\gamma)$ be a bottom-top tangle of type $(g,g)$ with $B$ a homology cube. Let us label its connected components by $\{1^+,\ldots,g^+\}\cup\{1^-,\ldots,g^-\}=:\left\lfloor g \right\rceil^+ \cup\left\lfloor  g\right\rceil^-$, where the label $k^{\pm}$ is assigned to the component $\gamma_k^{\pm}$. The \emph{linking matrix} of $(B,\gamma)$ is the matrix, with rows and columns indexed by  $\left\lfloor g \right\rceil^+ \cup\left\lfloor  g\right\rceil^-$, defined by
\begin{equation}\label{defilinkingbt}
\text{\emph{Lk}}_B(\gamma):= \text{\emph{Lk}}_{\hat{B}}(\hat{\gamma}),
\end{equation}
where $\hat{B}$ is the homology sphere $B\cup_b(\mathbb{S}^3\setminus \left[-1,1\right]^3)$ and $\hat{\gamma}$ is the framed oriented link in $\hat{B}$ whose component $\hat{\gamma}_j^{\pm}$ is obtained from $\gamma_j^{\pm}$ by connecting $p_j\times(\pm 1)$ with $q_j\times(\pm 1)$ with a small arc, and $\text{\emph{Lk}}_{\hat{B}}(\hat{\gamma})$ denotes the usual linking matrix of $\hat{\gamma}$ in the homology sphere  $\hat{B}$.
\end{definition}

Let $(M,m)\in\mathcal{C}$ and let $(B,\gamma)$ be its bottom-top tangle presentation. If $B$ is a homology cube, we  define the linking matrix $\text{Lk}(M)$ of $(M,m)$ as the linking matrix of its bottom-top tangle presentation. 

\subsection{Johnson homomorphisms for homology cobordisms}
The Johnson filtration and the Johnson homomorphisms of $\mathcal{M}$ extend in a natural way to the monoid of homology cobordisms, see  \cite{MR2131016}. Given $M=(M,m)$ in $\mathcal{C}$, since $m_+$ and $m_-$ induce isomorphisms in homology in all degrees,  by  Stallings' theorem  \cite[Theorem 3.4]{MR0175956}, the maps $m_{\pm,*}:\pi/\Gamma_k\pi\rightarrow\pi_1(M,*)/\Gamma_k\pi_1(M,*)$ are isomorphisms for all $k\geq 2$.  Hence, the  nilpotent version of the Dehn-Nielsen-Baer representation of the mapping class group  can be extended to  $\mathcal{C}$. For every positive integer $k$ define 
\begin{equation}\label{ecuacion2.3}
\rho_{k}:\mathcal{C}\longrightarrow \text{Aut}(\pi/\Gamma_{k+1}\pi),
\end{equation}
\noindent by sending $(M,m)\in\mathcal{C}$ to  the automorphism $m_{-,*}^{-1}\circ m_{+,*}$.

The \emph{Johnson filtration} $\{J_k\mathcal{C}\}_{k\geq1}$ of $\mathcal{C}$ is the descending chain of submonoids 
$$\mathcal{C}\supseteq J_1\mathcal{C}\supseteq J_2\mathcal{C}\supseteq\cdots\supseteq J_k\mathcal{C}\supseteq J_{k+1}\mathcal{C}\supseteq\cdots$$
where  $J_k\mathcal{C}:=\text{ker}(\rho_{k})$ for all $k\geq1$. The submonoid $J_1\mathcal{C}$ is denoted by $\mathcal{IC}$ and it is called the \emph{monoid of homology cylinders}.  Notice that under the embedding  described in Example \ref{ejemplo1}, the Torelli group $\mathcal{I}$ is mapped into $\mathcal{IC}$. Let $M=(M,m)\in\mathcal{C}_{g,1}$ and $(B,\gamma)$ be its bottom-tangle presentation. We have that $M$ belongs to $\mathcal{IC}$ if and only if  $B$ is a homology cube and $\text{Lk}(M)=\left( \begin{smallmatrix} 0&\text{Id}_g\\ \text{Id}_g&0 \end{smallmatrix} \right)$, see Lemma \ref{lemmalinking}.

For $k\geq 1$ the $k$-th \emph{Johnson homomorphism} for homology cobordisms
\begin{equation}\label{ecuacion2.4}
\tau_k: J_k\mathcal{C}\longrightarrow  H\otimes \mathfrak{L}_{k+1}(H),
\end{equation}
\noindent is defined as in the mapping class group case. In this case we also have that $\tau_k$ takes values in $D_k(H)$. We refer to \cite{MR2131016} for further details.

It was shown by S. Morita that  the Johnson homomorphism $\tau_k:J_k\mathcal{M}\rightarrow D_k(H)$ is not surjective in general  (see \cite[Section 6]{MR1224104}). The situation changes if we enlarge the mapping class group to the monoid of homology cobordisms. 
S. Garoufalidis and J. Levine proved in \cite[Theorem 3, Proposition 2.5]{MR2131016} the following.

\begin{theorem}[S. Garoufalidis, J. Levine]\label{thmgl} For every positive integer $k$, the $k$-th Johnson homomorphism $\tau_k:J_k\mathcal{C}\rightarrow D_k(H)$ is surjective.
\end{theorem}

Their proof uses obstruction theory and surgery techniques. N. Habegger gave in \cite{habegger2000milnor} a different proof of this theorem based on the surjectivity of Milnor invariants. We shall recall his proof in the next subsection, since it will be useful to us later.

\subsection{Milnor invariants and the Milnor-Johnson correspondence}\label{MJcorrespondence}
In this subsection we recall the Milnor invariants for string links and the Milnor-Johnson correspondence, which  relates the Johnson homomorphisms with the Milnor invariants. We refer to \cite{MR1026062,MR1620841}  for more details about Milnor invariants and to \cite{habegger2000milnor,MR2403806} for  more details about the Milnor-Johnson correspondence.

\subsubsection{String links and Milnor invariants}
 We start by introducing the definition of a string link in a homology cube. Denote by $D_l$ the surface $\Sigma_{0,1}$ together with  $l$ fixed  different points $p_1,\ldots,p_l$   distributed uniformly along the horizontal axis $\{(x,0)\ |\ x\in[-1,1]\}$, see Figure \ref{figura4.1}($a$). A \emph{string link on} $l$ \emph{strands} is an equivalence class of  pairs $(B,\sigma)$, where $B=(B,b)$  is a homology cube and $\sigma=(\sigma^{1},\ldots,\sigma^l):[-1,1]^l\rightarrow B$ is an oriented framed embedding such that $\sigma^i(\pm 1)=b(p_i,\pm 1)$, see Figure \ref{figura4.1}($b$).  Two pairs $(B,\sigma)$ and $(B',\sigma')$ are \emph{equivalent} if there exists an equivalence of homology cobordisms sending $\sigma$ to $\sigma'$.

The \emph{linking matrix} of a string link $(B,\sigma)$ on $l$ strands is the matrix, with rows and columns indexed by the components of $\sigma$, defined by
$$\text{Lk}_B(\sigma):= \text{Lk}_{\hat{B}}(\hat{\sigma}),$$
where $\hat{B}$ is the homology sphere $B\cup_b(\mathbb{S}^3\setminus \left[-1,1\right]^3)$ and $\hat{\sigma}$ is the \emph{braid closure} of $\sigma$, see Figure \ref{figura4.1}($c$).

\begin{figure}[ht!]
										\centering
                        \includegraphics[scale=0.58]{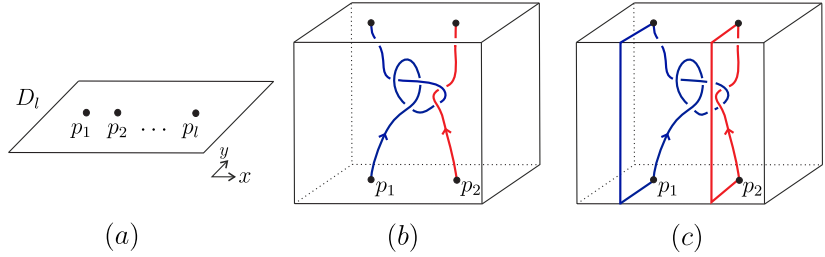}
												\caption{($a$) $D_l$, ($b$) a string link $\sigma$ on $2$ strands in $[-1,1]^2\times[-1,1]$ and ($c$) braid closure of $\sigma$.} 
												\label{figura4.1}
\end{figure}

By using the composition of homology cobordisms we can compose string links on $l$ strands. The equivalence class of $([-1,1]^3, \text{Id}_l)$, where $\text{Id}_l$ is the trivial string link, is the identity for this composition. Denote this monoid by $\mathcal{S}_l$.

We now turn to the definition of the Milnor invariants. Let $N(\{p_1,\ldots,p_l\})$ be a tubular neighborhood of the fixed points in $D_l$. Let  $D_l^o$ denote $D_l\setminus\text{int}(N(p_1,\ldots, p_l))$ and denote by $F_l$ the fundamental group $\pi_1(D_l^o, *)$ where $*\in\partial D_l$. We identify $F_l$ with the free group on $\{u_1,\ldots,u_l\}$, where $u_i$ is the homotopy class of a loop encircling the $i$-th hole of $D_g^o$ in the counterclockwise sense. Let $(B,\sigma)$ be a string link on $l$ strands. Set $S:=B\setminus \text{int}(N(\sigma))$, where $N(\sigma)$ is a tubular neighborhood of $\sigma$. The homeomorphism $b:\left[-1,1\right]^3\rightarrow\partial B$ and the framing of $\sigma$ determine an orientation-preserving homeomorphism $s:\partial(D_l^o\times\left[-1,1\right])\rightarrow\partial S$. Denote by $s_{\pm}:D_l^o\times\{\pm1\}\rightarrow\partial S$ the top and bottom restrictions of $s$. Since $B$ is a homology cube, the induced maps in homology $s_{\pm,*}:H_*(D_l^o;\mathbb{Z})\rightarrow H_*(S;\mathbb{Z})$ are isomorphisms. It follows from Stallings' theorem \cite[Theorem 3.4]{MR0175956} that $s_{\pm,*}$ induce isomorphisms on the nilpotent quotients of the fundamental groups. Thus we can define for every positive integer $k$, the $k$-th \emph{Artin representation} as the monoid homomorphism
\begin{equation}\label{ecuacion4.4}
A_k: \mathcal{S}_l\longrightarrow \text{Aut}\left(\frac{F_l}{\Gamma_{k+1}F_l} \right),
\end{equation}
that sends $(B,\sigma)$ to the automorphism $s_{-,*}^{-1}\circ s_{+,*}$. The \emph{Milnor filtration} of $\mathcal{S}_l$ is the descending chain of submonoids  
$$\mathcal{S}_l=\mathcal{S}_l[1]\supseteq\mathcal{S}_l[2]\supseteq\cdots\supseteq\mathcal{S}_l[k]\supseteq\mathcal{S}_l[k+1]\supseteq\cdots$$
where $\mathcal{S}_l[k]:=\text{ker} (A_k)$. Notice that  $\mathcal{S}_l[2]$ is the submonoid of string links with trivial linking matrix. 

Let $(B,\sigma)\in \mathcal{S}_l[k]$ and let $\lambda_i$ be the $i$-th longitude determined by the framing of the component $\sigma^i$. Since $(B,\sigma)\in \mathcal{S}_l[k]$, the homotopy class of the loop determined by $\lambda_i$ becomes trivial in $\pi_1(S)/\Gamma_{k}\pi_1(S)$. Therefore we can define the monoid homomorphism
$$ \mu_k: \mathcal{S}_l[k]\longrightarrow \frac{F_l}{\Gamma_2 F_l}\otimes\frac{\Gamma_k F_l}{\Gamma_{k+1}F_l}$$
by the formula
\begin{equation}\label{ecuacion4.5}
\mu_k(B,\sigma)=\sum_{i=1}^{l} u_i\otimes s_{-,*}^{-1}(\lambda_i).
\end{equation}

Let us identify $F_l/\Gamma_2 F_l$ with $\widetilde{H}:=H_1(D_l^o;\mathbb{Z})$ and $(\Gamma_{k} F_l)/(\Gamma_{k+1} F_l)$ with  the $k$-th term $\mathfrak{L}_{k}(\widetilde{H})$ of the free Lie algebra generated by $\widetilde{H}$. The fact that  the Artin representation fixes the homotopy class of $\partial D_l$ implies that $\mu_k$ takes values in the kernel $D_{k-1}(\widetilde{H})$ of the Lie bracket   $\left[\ ,\ \right]: \widetilde{H}\otimes \mathfrak{L}_{k}(\widetilde{H})\rightarrow\mathfrak{L}_{k+1}(\widetilde{H})$. From the above discussion, for all $k\geq 2$ we can write
\begin{equation}\label{ecuacion4.6}
\mu_k: \mathcal{S}_l[k]\longrightarrow D_{k-1}(\widetilde{H}).
\end{equation}

The monoid homomorphism $\mu_k$ is called the $k$-th \emph{Milnor map}. Notice that $\text{ker}(\mu_k)=\mathcal{S}_l[k+1]$.  In \cite[Section 1]{MR1620841} N. Habegger and X. Lin proved   that  for all $k\geq1$ the $k$-th  Milnor map $\mu_k$ is surjective. The idea of their proof was adapted from the work of  K. Orr  in  \cite{MR974908}, where he studied which Milnor invariants are realizable. There is a more geometric approach to the realizability of Milnor invariants developed by  T. Cochran in \cite{MR1042041, MR1055569}, which we will need and  sketch briefly in  subsection \ref{section4.4}.

\subsubsection{The Milnor-Johnson correspondence}
In \cite{habegger2000milnor}, N. Habegger defined a bijection between homology cylinders and string links with trivial linking matrix. We follow the construction in \cite{MR2403806} which  can be described schematically as follows:
\begin{equation}\label{procedure}
\text{homology cylinder } \rightsquigarrow \text{bottom-top tangle }  \rightsquigarrow    \text{string  link.}
\end{equation}

More precisely, let $(M,m)$ be a homology cylinder over $\Sigma_{g,1}$ and consider its bottom-top tangle presentation. Next, from a bottom-top tangle of type $(g,g)$ we can obtain a string on $2g$ strands by the method illustrated in Figure \ref{figura4.8}. 
\begin{figure}[ht!] 
										\centering
                        \includegraphics[scale=0.7]{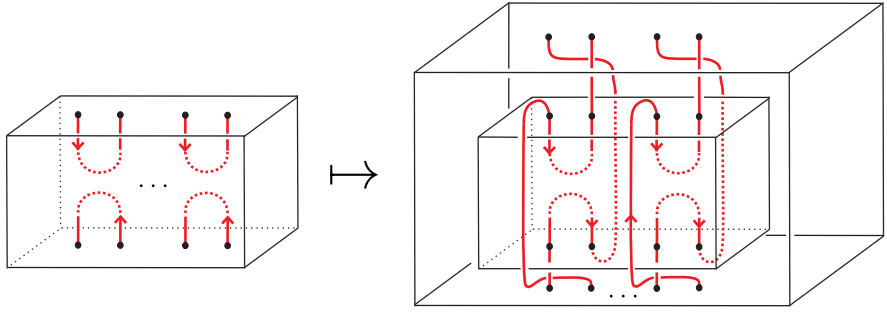}
												\caption{From  bottom-top tangles to string links.}
												\label{figura4.8}
\end{figure}

In this way we transform a homology cylinder $(M,m)\in\mathcal{IC}_{g,1}$ into a string link $\text{MJ}(M)\in\mathcal{S}_{2g}$. N. Habegger proved in \cite{habegger2000milnor} that $\text{MJ}$ defines a bijection between $\mathcal{IC}_{g,1}$ and the submonoid $\mathcal{S}_{2g}[2]$ of string links with trivial linking matrix. Moreover, for all $k\geq1$ the following diagram is commutative (see \cite[Claim 8.16]{MR2403806}).
\begin{equation}\label{ecuacion4.9}
\xymatrix{ J_k\mathcal{C}_{g,1}\ar[r]^{\text{MJ}\ \ }_{\cong\ \ }\ar[d]_{\tau_{k}} & \mathcal{S}_{2g}[k+1] \ar[d]^{\mu_{k+1}} \\
						D_k(H)\ar[r]^{\cong\ \ } & D_k(\widetilde{H}),}
\end{equation}
where the bottom isomorphism is induced by the identification $\pi\cong F_{2g}$ described as follows. Consider  a free basis $\{\alpha_1,\ldots,\alpha_g,\beta_1,\ldots,\beta_g\}$ of $\pi$ induced by basing at $*$ the system of meridians and parallels in Figure \ref{figura3.0}. Identify  $\alpha_i$ with $u^{-1}_{2i-1}$ and $\beta_i$ with $u_{2i}$. 

In this way, from the surjectivity of $\mu_{k+1}$ and diagram (\ref{ecuacion4.9}),  it follows that $\tau_k:J_k\mathcal{C}\rightarrow D_k(H)$ is surjective. This is the proof of Theorem \ref{thmgl} by N. Habegger \cite{habegger2000milnor}. 

\subsection{Diagrammatic version of the Johnson homomorphisms}\label{subsection2.5}

In order to relate the Kontsevich integral with the Milnor invariants, N. Habegger and G. Masbaum gave in \cite{MR1783857} a diagrammatic version of the Milnor map. This was also  done for Johnson homomorphisms by S. Garoufalidis and J. Levine  in \cite{MR2131016}. Let us recall this description.

By a \emph{tree-like Jacobi diagram} we mean a finite contractible unitrivalent graph such that the trivalent vertices are \emph{oriented}, that is, each set of incident edges to a trivalent vertex is endowed with a cyclic order. The \emph{internal degree} of such a diagram is the number of trivalent vertices; we denote it by i-deg.
Let $C$ be a finite set. We say that a tree-like Jacobi diagram $T$ is   $C$-\emph{colored} if there is a map from the set of univalent vertices (\emph{legs}) of $T$ to the free abelian group generated by $C$. We use dashed lines to represent tree-like Jacobi diagrams and, when we draw them, we assume that the orientation of trivalent vertices is counterclockwise.

Consider the  abelian group
$$\mathcal{T}(C):=\frac{\mathbb{Z}\{C\text{-colored tree-like Jacobi diagrams}\}}{\text{AS, IHX, $\mathbb{Z}$-multilinearity}},$$
\noindent where the relations AS, IHX are local  and the multilinearity relation applies to the $C$-colored legs, see Figure \ref{figura4.13}.
\begin{figure}[ht!] 
										\centering
                        \includegraphics[scale=0.72]{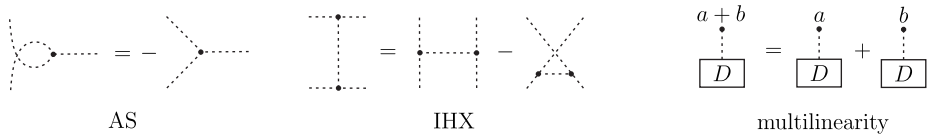}
												\caption{Relations in $\mathcal{T}(C)$. Here $a,b\in \mathbb{Z}\cdot C$.}
												\label{figura4.13}
\end{figure}

 \noindent Notice that $\mathcal{T}(C)$ is graded by the internal degree: for $k\geq 1$, $\mathcal{T}_k(C)$ is the subspace of $\mathcal{T}(C)$ generated by tree-like Jacobi diagrams of i-deg $=k$. We can define $\mathcal{T}(G)$ for any finitely generated free abelian group $G$ by  $\mathcal{T}(G)=\mathcal{T}(C)$ where $C$ is any set of free generators of $G$. 

Consider the abelian group $H=H_1(\Sigma_{g,1};\mathbb{Z})$. We have seen  that the $k$-th Johnson homomorphism takes values in $D_k(H)\subseteq H\otimes\mathfrak{L}_{k+1}(H)$. Observe that a \emph{rooted} tree of i-deg $=k$ with $H$-colored legs determines a Lie commutator in $\mathfrak{L}_{k+1}(H)$. Let us consider the map 
\begin{equation}\label{ecuacion4.15tree}
\eta^{\mathbb{Z}}_k:\mathcal{T}_k(H)\longrightarrow D_k(H),\ \ \ \ T\longmapsto\sum_v \text{color}(v)\otimes (T \text{ rooted at } v),
\end{equation}
where the sum ranges over the set of univalent vertices of $T$, and the rooted trees are identified with Lie commutators. For instance,

\bigskip

\centerline{\includegraphics[scale=0.85]{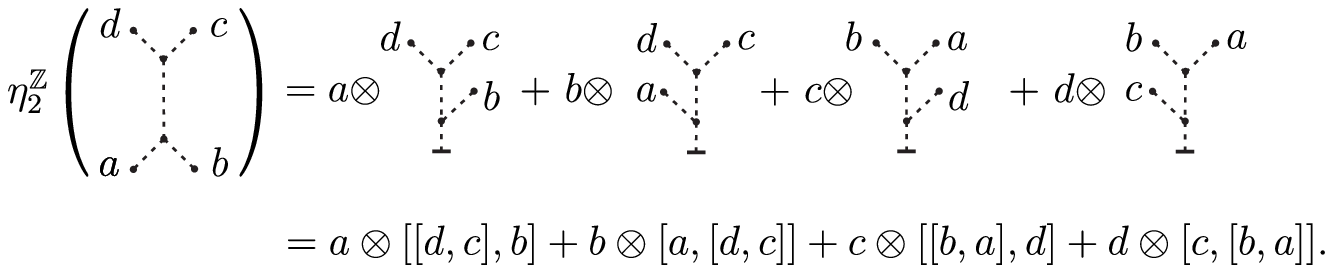}}

\medskip

Consider the rational version of  $\eta^{\mathbb{Z}}_k$: 
\begin{equation}\label{ecuacion4.15treerat}
\eta_k:\mathcal{T}_k(H)\otimes\mathbb{Q}\longrightarrow D_k(H)\otimes\mathbb{Q}.
\end{equation}
This map  is an isomorphism, see \cite[Corollary 3.2]{MR1943338}.  In this way, for $M\in J_k\mathcal{C}_{g,1}$ we define the \emph{diagrammatic version} of the $k$-th Johnson homomorphism by
$$\eta_k^{-1}(\tau_k(M))\in \mathcal{T}_k(H)\otimes\mathbb{Q}.$$

\section{Lagrangian version of the Johnson homomorphisms}\label{seccion3}

In \cite{MR1823501,MR2265877}, J. Levine introduced a different filtration of the mapping class group by considering a handlebody bounded by $\Sigma$. The induced inclusion determines a  Lagrangian subgroup of the first homology group of the surface. This Lagrangian subgroup, together with the lower central series of the fundamental group of the handlebody, allow to define the new filtration. 

\subsection{Preliminaries}\label{seccion3.1} Let $V$ (or $V_g$ if there is ambiguity) be a handlebody of genus $g$. Fix a disk $D$ on the boundary of $V$ such that $\partial V=\Sigma\cup D$, where $D$ and $\Sigma$ are glued along their boundaries. Denote by  $\iota$ the inclusion of $\Sigma$ into $\partial V\subseteq V$, see Figure \ref{figura17}. Set $H':=H_1(V;\mathbb{Z})$ and $\pi':=\pi_1(V,\iota(*))$. Denote by $A$ the kernel of the induced map $\iota_*:H\rightarrow H'$ in homology and by  $\mathbb{A}$ the kernel of the induced map $\iota_{\#}:\pi\rightarrow\pi'$ in homotopy. Notice that $A$ is a \emph{Lagrangian}  subgroup of $H$  with respect to the intersection form $\omega:H\otimes H\rightarrow\mathbb{Z}$.

Let us denote by $\text{ab}:\pi\rightarrow H$ and $\text{ab}':\pi'\rightarrow H'$ the abelianization maps. The equality $\iota_*\circ \text{ab}=\text{ab}'\circ\iota_{\#}$ implies that  $\text{ab}^{-1}(A)=\mathbb{A}\cdot\Gamma_2\pi$. Thus, we have 
\begin{equation*}
A \mathrel{\mathop{\longleftarrow}^{\mathrm{\cong}}_{\text{ab}}} (\mathbb{A}\cdot\Gamma_2\pi)/\Gamma_2\pi\cong \mathbb{A}/(\Gamma_2\pi\cap\mathbb{A}).
\end{equation*}
By Hopf's formula, we obtain
$(\Gamma_2\pi\cap\mathbb{A})/[\pi,\mathbb{A}]\cong H_2(\pi/\mathbb{A})\cong H_2(\pi'),$
\noindent and since $\pi'$ is a free group, $H_2(\pi')=0$. Hence $\Gamma_2\pi\cap\mathbb{A}=[\pi,\mathbb{A}]$. To sum up, we have the short exact sequence
\begin{equation}\label{ecuacion3.3}
1\longrightarrow [\pi,\mathbb{A}]\longrightarrow \mathbb{A}\stackrel{\text{ab}}{\longrightarrow}A\longrightarrow 1.
\end{equation}

Finally, let us recall the symplectic representation. Since the elements in $\mathcal{M}$ are orientation-preserving, their induced maps on $H$ preserve the intersection form. Therefore we have a map
$$\mathcal{M}\longrightarrow\text{Sp}(H)=\{f\in\text{Aut}(H)\ \ |\ \ \forall x,y\in H,\ \omega(f(x),f(y))=\omega(x,y) \},$$
that sends $h\in\mathcal{M}$ to the induced map $h_*$ on $H$. We will often need to consider bases to perform some computations. On this purpose, we fix  a free basis $\{\alpha_1,\ldots,\alpha_g,\beta_1,\ldots,\beta_g\}$ induced by basing at $*$ the fixed system of meridians and parallels in Figure \ref{figura3.0}. We also fix the   symplectic basis $\{a_1,\ldots,a_g, b_1, \ldots, b_g\}$ of $H$, induced by $\{\alpha_1,\ldots,\alpha_g,\beta_1,\ldots,\beta_g\}$. Here we assume that the curves $\iota(\alpha_i)$'s bound disks in $V$, see Figure \ref{figura17}. This way,  $\{a_1,\ldots,a_g\}$ is a basis for $A$,  and $\{b_1+A,\ldots, b_g+A\}$ is a basis for $H/A$. 
\begin{figure}[ht!]
										\centering
                        \includegraphics[scale=0.75]{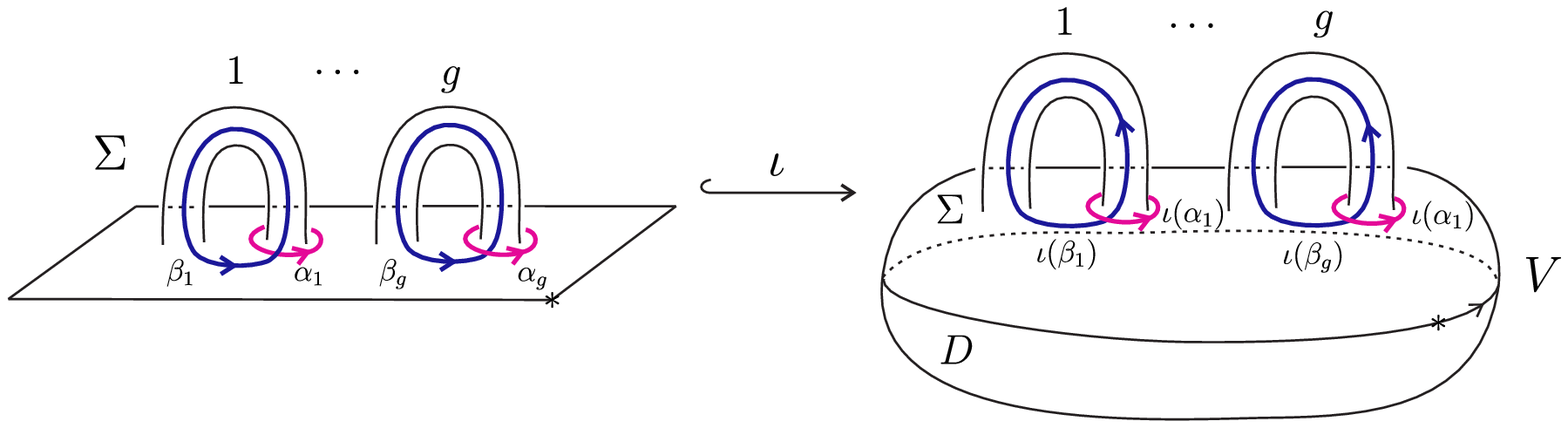}
												\caption{The inclusion $\Sigma\stackrel{\iota}{\hooklongrightarrow}V$.}
												\label{figura17}
\end{figure}

We use the above symplectic basis of $H$ to identify $\text{Sp}(H)$ with the group  $\text{Sp}(2g,\mathbb{Z})$ of $(2g)\times(2g)$ matrices $\Lambda$ with integer entries such that  ${\Lambda}^TJ\Lambda=J$, where $J$ is the standard invertible skew-symmetric matrix $\left( \begin{smallmatrix} 0&\text{Id}_g\\ -\text{Id}_g&0 \end{smallmatrix} \right)$.  Denote this identification by $\psi:\text{Sp}(H)\rightarrow \text{Sp}(2g,\mathbb{Z})$.

\subsection{The Lagrangian mapping class group}\label{seccion3.2}

Let us define two important subgroups of the mapping class group associated to the Lagrangian subgroup $A$. Set
\begin{equation}\label{ecuacion3.4}
\mathcal{L}:=\{f\in\mathcal{M}\ \ |\ \ f_{*}(A)\subseteq A\}  \text{\ \ \ \ \ \ \ and \ \ \ \ \ \ }  \mathcal{IL}:=\{f\in\mathcal{L}\ \ |\ \ {f_*}|_A=\text{Id}_A\}.
\end{equation}

\noindent The subgroup $\mathcal{L}$ is called   the \emph{Lagrangian mapping class group} of $\Sigma$,  and $\mathcal{IL}$ is called the  \emph{strongly Lagrangian mapping class group} of $\Sigma$. 

\begin{example} The Torelli group $\mathcal{I}$ is contained in $\mathcal{IL}$. Also, any Dehn twist along a meridian $\alpha_i$ (see Figure \ref{figura3.0}) belongs to $\mathcal{IL}$. This shows that $\mathcal{I}$ is strictly contained in $\mathcal{IL}$.
\end{example}

\begin{example} Let $g\geq2$. Consider the orientation-preserving homeomorphism $h:\Sigma \rightarrow \Sigma$ that interchanges the first and second handle in Figure \ref{figura3.0}. This homeomorphism   can be extended to the handlebody $V$ (this extension is known in the literature as \emph{interchanging two knobs}, see \cite[Section 3]{MR0433433} for a detailed description). We have  that $h$ belongs to $\mathcal{L}$ but not to $\mathcal{IL}$, hence $\mathcal{IL}$ is strictly contained in $\mathcal{L}$.  
\end{example}

Let us give some equivalent formulations of the strongly Lagrangian mapping class group. If $h$ belongs to $\mathcal{L}$, then it induces a well defined isomorphism $\hat{h}_*:H/A\rightarrow H/A$ by sending $x+A$ to $h_{*}(x)+A$. By means of  the isomorphism $H/A\stackrel{\iota_*}{\longrightarrow}H'$, we have an isomorphism $h'_*:H'\rightarrow H'$ defined by $h'_*:=\iota_*\circ \hat{h}_*\circ \iota_*^{-1}$.

\begin{lemma}\label{lema3.1} Let $h$ be an element in $\mathcal{L}$. The following assertions are equivalent:
\begin{enumerate}
\item[\emph{(\emph{i})}] $h$ belongs to $\mathcal{IL}$.
\item[\emph{(\emph{ii})}] The induced isomorphism $h'_*:H'\rightarrow H'$ is the identity.
\item[\emph{(\emph{iii})}] $\iota_*\circ h_*=\iota_*$.
\end{enumerate}
\end{lemma} 

\begin{proof}
We use the symplectic basis $\{a_1,\ldots,a_g, b_1, \ldots, b_g\}$ and the  identification  $\psi: \text{Sp}(H)\rightarrow \text{Sp}(2g,\mathbb{Z})$ described in subsection \ref{seccion3.1}. Let  $h\in\mathcal{L}$, then there exist integers $\lambda_{kj}$, $\delta_{kj}$ and $\epsilon_{kj}$   such that for  $1\leq j \leq g$,  
\begin{equation}\label{ecuacionlema}
h_{*}(a_j)=\sum_{k=1}^g \lambda_{kj}a_k  \text{\ \ \ \ \ \ \ and \ \ \ \ \ \ }  h_{*}(b_j)=\sum_{k=1}^g \delta_{kj}a_k + \sum_{k=1}^g \epsilon_{kj}b_k.
\end{equation}

\noindent Hence  $\psi(h_{*})=\left( \begin{smallmatrix} P & Q\\ 0 & R \end{smallmatrix} \right)$, where $P=(\lambda_{ij})$, $Q=(\delta_{ij})$ and $R=(\epsilon_{ij})$. The symplectic condition on $h_{*}$ becomes
\begin{equation}\label{ecuacion3.5}
P^TR=\text{Id}_g  \text{\ \ \ \ \ \ \ and \ \ \ \ \ \ } Q^TR=R^TQ.
\end{equation}
\noindent Recall that $\{a_1,\ldots,a_g\}$ is a basis for $A$, and $\{b_1+A,\ldots, b_g+A\}$ is a basis for $H/A$. The matrices of $h_{*}|_A:A\rightarrow A$ and $\hat{h}_*:H/A\rightarrow H/A$ in these bases are $P$ and $R$ respectively. The first condition in  (\ref{ecuacion3.5}) implies that $P=\text{Id}_g$ if and only if $R=\text{Id}_g$. Therefore we have the equivalence ($i$)$\Leftrightarrow$($ii$). Now, from the definition of $h'_*$  it follows that  $h'_*=\text{Id}_{H'}$ if and only if $\iota_*\circ \hat{h}_*=\iota_*$ on $H/A$ if and only if $\iota_*\circ h_*=\iota_*$ on $H$. Hence we have ($ii$)$\Leftrightarrow$($iii$).
\end{proof}

We now describe the filtration introduced by J. Levine in \cite{MR1823501,MR2265877}.

\begin{definition}\label{deflagfil} The \emph{Lagrangian filtration} or \emph{Johnson-Levine filtration}  $\{J_k^L\mathcal{M}\}_{k\geq1}$ of  $\mathcal{M}$ is defined as
$$J_k^L\mathcal{M}:=\{h\in\mathcal{M}\ \ |\ \ \iota_{\#}h_{\#}(\mathbb{A})\subseteq\Gamma_{k+1}\pi',\  {h_*}|_A=\text{\emph{Id}}_A\}.$$
\end{definition}

Notice that the condition $\iota_{\#}h_{\#}(\mathbb{A})\subseteq\Gamma_{k+1}\pi'$  implies that $h_{*}(A)\subseteq A$. Besides, J. Levine also defined and studied in \cite{MR1823501,MR2265877} a version of the Johnson homomorphisms for the above filtration.  In order to define them, let us first  identify $H/A$ with $A^*$ by sending $x+A\in H/A$ to  $\omega(x,\cdot)\in A^*$. We also identify $H/A$ with $H'$ via  the isomorphism $\iota_*$.

\begin{proposition}(J. Levine) For every non-negative integer $k$, the $k$-th term $J_k^L\mathcal{M}$ of the Johnson-Levine filtration is a subgroup of $\mathcal{M}$. Let 
\begin{equation}\label{ecuacion3.9}
\tau_k^L:J_k^L\mathcal{M}\rightarrow\text{\emph{Hom}}(A,\Gamma_{k+1}\pi'/\Gamma_{k+2}\pi')\cong A^*\otimes \Gamma_{k+1}\pi'/\Gamma_{k+2}\pi'\cong H'\otimes \mathfrak{L}_{k+1}(H'),
\end{equation}

\noindent be the map that sends $h\in J_k^L\mathcal{M}$ to the map $a\in A\mapsto \{\iota_{\#}h_{\#}(\alpha)\}_{k+2}$, where $\alpha\in \mathbb{A}$ is such that $\text{\emph{ab}}(\alpha)=a$. Then $\tau_k^L$ is a group homomorphism which we shall call the $k$-th \emph{Johnson-Levine homomorphism}.
\end{proposition}

For the sake of completeness let us see the proof.

\begin{proof}  The argument is by induction on $k$. From Definition \ref{deflagfil}, it follows that $J_1^L\mathcal{M}=\mathcal{IL}$ which is indeed a subgroup of $\mathcal{M}$. Now suppose that $J_k^L\mathcal{M}$ is a subgroup. Let us verify that $\tau_k^L$ is well defined and it is a group homomorphism. Let $h\in J_k^L\mathcal{M}$, $a\in A$ and $\alpha_1,\alpha_2\in \mathbb{A}$ such that $\text{ab}(\alpha_1)=\text{ab}(\alpha_2)=a$. The short exact sequence (\ref{ecuacion3.3}) implies $\alpha_1\alpha_2^{-1}=[x_1,y_1]\cdots[x_n,y_n]$ with $x_1,\ldots,x_n\in\pi$ and $y_1,\ldots,y_n\in\mathbb{A}$. Since for every $j$, we have that $[\iota_{\#}h_{\#}(x_j),\iota_{\#}h_{\#}(y_j)]$ is in $[\pi',\Gamma_{k+1}\pi']=\Gamma_{k+2}\pi'$, then $\iota_{\#}h_{\#}(\alpha_1\alpha_2^{-1})$ belongs to $\Gamma_{k+2}\pi'$, so $\tau_k^L(h)$ is well defined as a map from $H$ to $\Gamma_{k+1}\pi'/\Gamma_{k+2}\pi'$. 

Clearly $\tau_k^L(h)$ belongs to $\text{Hom}(A,\Gamma_{k+1}\pi'/\Gamma_{k+2}\pi')$. Let us see that $\tau_k^L$ is a group homomorphism. Let $h,\widetilde{h}\in J_k^L\mathcal{M}$, $a\in A$ and $\alpha\in\mathbb{A}$ with $\text{ab}(\alpha)=a$.
The splittable short exact sequence 
\begin{equation}\label{ecuacion3.10}
1\longrightarrow \mathbb{A}\longrightarrow \pi\stackrel{\iota_{\#}}{\longrightarrow}\pi'\longrightarrow 1,
\end{equation}
\noindent and the fact that $\widetilde{h}\in J_k^L\mathcal{M}$, allow us to write $\widetilde{h}_{\#}(\alpha)=\beta y$ with  $\beta\in\mathbb{A}$ and $y\in\Gamma_{k+1}\pi$.  Notice that $a=\widetilde{h}_{*}(a)=\text{ab}(\widetilde{h}_{\#}(\alpha))=\text{ab}(\beta y)=\text{ab}(\beta)$.

On the other hand,  suppose  that $y$ is  a group commutator of length $k+1$, say in the elements $y_1,\ldots, y_{k+1}\in\pi$ (if $y$ is a product of such commutators, the reasoning is similar). Then $\iota_{\#}(h_{\#}(y))$ and $\iota_{\#}(y)$ are commutators of length $k+1$ in the elements  $\iota_{\#}(h_{\#}(y_1))$,$\ldots$, $\iota_{\#}(h_{\#}(y_{k+1}))$ and $\iota_{\#}(y_1),\ldots, \iota_{\#}(y_{k+1})$ respectively. Notice that $y$, $\iota_{\#}(y)$ and $\iota_{\#}(h_{\#}(y))$ have the same  \emph{commutator structure}, that is, they have the same bracketing structure.

Under the identification $\Gamma_{k+1}\pi'/\Gamma_{k+2}\pi'\cong \mathfrak{L}_{k+1}(H')$, the elements $\iota_{\#}(h_{\#}(y))\Gamma_{k+2}\pi'$ and $\iota_{\#}(y)\Gamma_{k+2}\pi'$ correspond to Lie commutators,  with the same structure as $y$, in the elements $\text{ab}'(\iota_{\#}(h_{\#}(y_1)))$,$\ldots$, $\text{ab}'(\iota_{\#}(h_{\#}(y_{k+1})))$ and $\text{ab}'(\iota_{\#}(y_1)),\ldots, \text{ab}'(\iota_{\#}(y_{k+1}))$ respectively. The identity $\iota_*\circ \text{ab}=\text{ab}'\circ \iota_{\#}$ and  Lemma \ref{lema3.1}($iii$)  imply that
$$\text{ab}'(\iota_{\#}h_{\#}(y_j))=\iota_{*}h_{*}(\text{ab}(y_j))=\iota_{*}(\text{ab}(y_j))=\text{ab}'(\iota_{\#}(y_j)),$$
\noindent thus $\iota_{\#}h_{\#}(y)\Gamma_{k+2}\pi'=\iota_{\#}(y)\Gamma_{k+2}\pi'$.  From the above discussion, it follows that 
\begin{equation}\label{ecuacion3.11}
\begin{split}
\tau_k^L(h\circ\widetilde{h})(a) & = \iota_{\#}(h_{\#}(\widetilde{h}_{\#}(\alpha)))\Gamma_{k+2}\pi'\\
																 & = \iota_{\#}(h_{\#}(\beta))\Gamma_{k+2}\pi'+\iota_{\#}(h_{\#}(y))\Gamma_{k+2}\pi'\\
																 & = \tau_k^L(h)(\text{ab}(\beta))+\iota_{\#}(h_{\#}(y))\Gamma_{k+2}\pi'\\
																 & = \tau_k^L(h)(a)+\iota_{\#}(y)\Gamma_{k+2}\pi'\\
																 & = \tau_k^L(h)(a)+\tau_k^L(\widetilde{h})(a),
\end{split}
\end{equation}
which shows that $\tau_k^L$ is a group homomorphism. From the definition of $\tau_k^L$ it follows that $\text{ker}(\tau_k^L)=J_{k+1}^L\mathcal{M}$, and so $J_{k+1}^L\mathcal{M}$ is a subgroup of $\mathcal{M}$. This completes the proof.
\end{proof}

 A similar argument to the one used to show that $\tau_k$ takes values in $D_k(H)$ \cite[Remark 3.3]{MR1224104}, works to show that $\tau_k^L$ takes values in
\begin{equation}\label{ecuacion3.12}
D_k(H'):=\text{ker}\left(\left[\ ,\ \right]:H'\otimes\mathfrak{L}_{k+1}(H')\longrightarrow\mathfrak{L}_{k+2}(H')\right).
\end{equation}
This was already remarked by J. Levine \cite[Proposition 4.3]{MR1823501}. Let us recall the argument.  Consider the bases fixed in subsection \ref{seccion3.1}. Then, for  $h\in J_k^L\mathcal{M}$  the $k$-th Johnson-Levine homomorphism can be written
\begin{equation}\label{ecuacion3.13}
\tau_k^L(h)=-\sum_{j=1}^{g}\iota_{*}(b_j)\otimes\{\iota_{\#}(h_{\#}(\alpha_j))\}_{k+2}=-\sum_{j=1}^{g}\iota_{*}(h_{*}(b_j))\otimes\{\iota_{\#}(h_{\#}(\alpha_j))\}_{k+2},
\end{equation}
where the second equality follows from Lemma \ref{lema3.1}($iii$).

\noindent  The Lie bracket  $\left[\ ,\ \right]:H'\otimes\mathfrak{L}_{k+1}(H')\longrightarrow\mathfrak{L}_{k+2}(H')$ corresponds to the commutator map 
\begin{equation}\label{ecuacion3.14}
\Psi:\frac{\pi'}{\Gamma_2\pi'}\otimes\frac{\Gamma_{k+1}\pi'}{\Gamma_{k+2}\pi'}\longrightarrow\frac{\Gamma_{k+2}\pi'}{\Gamma_{k+3}\pi'},
\end{equation}
that sends $\{x'\}_{2}\otimes\{y'\}_{k+2}$ to $[x',y']\Gamma_{k+3}\pi'$. Thus
\begin{align*}\label{ecuacion3.15}
\Psi(\tau_k^L(h))&=\sum_{j=1}^g\Psi\Big(\iota_{*}(h_{*}(-b_j))\otimes\iota_{\#}(h_{\#}(\alpha_j))\Gamma_{k+2}\pi'\Big)\\
								 &=\sum_{j=1}^g\Psi\Big(\{\iota_{\#}h_{\#}(\beta^{-1}_j)\}_2\otimes\{\iota_{\#}h_{\#}(\alpha_j)\}_{k+2}\Big)\\
								 &=\iota_{\#}h_{\#}\Big(\prod_{j=1}^g \left[\beta^{-1}_j,\alpha_j\right]\Big)\Gamma_{k+3}\pi'\\
								 &=\Gamma_{k+3}\pi',
\end{align*}
where the last equality holds because $\prod_{j=1}^g \left[\beta^{-1}_j,\alpha_j\right]$ represents the inverse of the homotopy class of $\partial\Sigma$ (see Figure \ref{figura17}), and this element is fixed by $h_{\#}$. Hence $\tau_k^L(h)\in D_k(H')$. To sum up, we have a  descending chain  of subgroups
\begin{equation}\label{ecuacion3.8}
\mathcal{M}\supseteq\mathcal{L}\supseteq \mathcal{IL}=J_1^L\mathcal{M}\supseteq J_2^L\mathcal{M}\supseteq\cdots
\end{equation} 
and a family of group homomorphisms 
\begin{equation}\label{ecuacion3.16}
\tau_k^L:J_k^L\mathcal{M}\rightarrow D_k(H').
\end{equation}

\begin{remark}\label{handlebodygroup}
Let $\mathcal{H}$ (or $\mathcal{H}_{g,1}$) be the subgroup of $\mathcal{M}$ consisting of the elements that can be extended to the handlebody $V$. The subgroup $\mathcal{H}$ is called the \emph{handlebody group} and is contained in $\mathcal{L}$. By virtue of Dehn's lemma $\mathcal{H}=\{h\in\mathcal{M}\ |\ h_{\#}(\mathbb{A})\subseteq \mathbb{A}\}$, see \cite[Theorem 10.1]{MR0159313}. J. Levine showed in \cite[Proposition 4.1]{MR1823501} that 
$$\bigcap_{k\geq1} J_k^L\mathcal{M}=\mathcal{H}\cap\mathcal{IL}.$$
The inclusion $\mathcal{H}\cap\mathcal{IL}\subseteq \bigcap_{k\geq1} J_k^L\mathcal{M}$ is clear. Now, let $h\in\bigcap_{k\geq1} J_k^L\mathcal{M}$ and $\alpha\in \mathbb{A}$, thus $\iota_{\#}h_{\#}(\alpha)\in\Gamma_{k+1}\pi'$ for all $k\geq1$. Since $\pi'$ is residually nilpotent we have $\iota_{\#}h_{\#}(\alpha)=1$, that is, $h_{\#}(\alpha)\in\mathbb{A}$. Hence $\bigcap_{k\geq1} J_k^L\mathcal{M}\subseteq\mathcal{H}\cap\mathcal{IL}$.
\end{remark}
\subsection{The monoid of Lagrangian homology cobordisms}

The Johnson-Levine filtration can be defined similarly on the monoid $\mathcal{C}$ of homology cobordisms. Let us start by defining the analogues to the Lagrangian and strongly Lagrangian mapping class groups.

The monoid of \emph{Lagrangian homology cobordisms} is defined as
\begin{equation}\label{ecuacion3.20}
\mathcal{LC}:=\{(M,m)\in\mathcal{C}\ \ |\ \ \rho_1(M)(A)\subseteq A\}=\{(M,m)\in\mathcal{C}\ \ |\ \ m_{+,*}(A)\subseteq m_{-,*}(A)\},
\end{equation}
and the monoid of \emph{strongly Lagrangian homology cobordisms} is defined as
\begin{equation}\label{ecuacion3.21}
\mathcal{ILC}:=\{(M,m)\in\mathcal{LC}\ \ |\ \ \rho_1(M)|_A=\text{Id}_A\}=\{(M,m)\in\mathcal{LC}\ \ |\ \ m_{+,*}|_A = m_{-,*}|_A\}.
\end{equation}

Notice that we have the inclusions $\mathcal{IL}\subseteq\mathcal{ILC}\subseteq\mathcal{LC}$. Let us see how these monoids are characterized in terms of the linking matrix.

\begin{lemma}\label{lemmalinking} Let $M\in\mathcal{C}_{g,1}$ and let $(B,\gamma)$ be its bottom-top tangle presentation. Then
\begin{enumerate}
\item[\emph{(\emph{i})}] $M$ belongs to $\mathcal{LC}_{g,1}$ if and only if $B$ is a homology cube and $\text{\emph{Lk}}(M)=\left( \begin{smallmatrix} 0 & \Lambda \\ {\Lambda}^{T} & \Delta \end{smallmatrix} \right)$, 

\medskip

\item[\emph{(\emph{ii})}] M belongs to $\mathcal{ILC}_{g,1}$ if and only if $B$ is a  homology cube and   $\text{\emph{Lk}}(M)=\left( \begin{smallmatrix} 0 & {\text{\emph{Id}}_g}\\ {\text{\emph{Id}}_g} & {\Delta} \end{smallmatrix} \right)$,

\medskip

\item[\emph{(\emph{iii})}] $M$ belongs to $\mathcal{IC}_{g,1}$ if and only if $B$ is a homology cube and   $\text{\emph{Lk}}(M)=\left( \begin{smallmatrix} 0 & {\text{\emph{Id}}_g}\\ {\text{\emph{Id}}_g} & 0 \end{smallmatrix} \right)$,
\end{enumerate}
where  $\Lambda$ and $\Delta$ are $g\times g$ matrices and $\Delta$ is symmetric.
\end{lemma}
\begin{proof}

The proof is similar to the proof of \cite[Lemma 2.12]{MR2403806}. Consider the bases fixed  in subsection \ref{seccion3.1}. A Mayer-Vietoris  argument shows that $H_1(B;\mathbb{Z})$ is isomorphic to the quotient of $H_1(M;\mathbb{Z})$ by the subgroup spanned by $S=\{m_{+,*}(b_1),\ldots m_{+,*}(b_g), m_{+,*}(a_1),\ldots, m_{+,*}(a_g)\}$. Hence  $B$ is a homology cube if and only if $S$ is a basis for $H_1(M;\mathbb{Z})$.

If $S$ is a basis  for $H_1(M;\mathbb{Z})$, then for $1\leq j\leq g$,    
\begin{equation}\label{eq1lemma}
m_{+,*}(a_j)= \sum_{k=1}^{g}\chi_{kj} m_{+,*}(b_k)+ \lambda_{kj} m_{-,*}(a_k) \text{\ \ \ and \ \ \ } m_{-,*}(b_j)= \sum_{k=1}^{g}\epsilon_{kj} m_{+,*}(b_k) + \delta_{kj} m_{-,*}(a_k),
\end{equation}
where the $\chi$'s, $\lambda$'s, $\epsilon$'s and $\delta$'s are integer coefficients. We observe that $m_-(\beta_k)$ and $m_-(\alpha_k)$ are  the oriented longitude and oriented meridian of $\gamma_k^-$, respectively. Similarly  $m_+(\alpha_k)$ and $m_+(\beta_k)$ are the oriented longitude and oriented meridian of $\gamma_k^+$, respectively,  see Figure \ref{figura4.7}. Now, the columns of $\text{Lk}(B,\gamma)$ express how the oriented longitudes $m_+(\alpha_1)$, $\ldots$, $m_+(\alpha_k)$, $m_-(\beta_1)$, $\ldots$, $m_-(\beta_g)$ expand in the basis $S$. So we have $\chi_{ij}=\chi_{ji}$, $\lambda_{ij}=\epsilon_{ji}$, $\delta_{ij}=\delta_{ji}$ and  $\text{Lk}(M)=\left( \begin{smallmatrix} X & \Lambda^T\\ \Lambda & \Delta \end{smallmatrix} \right)$, where $X=(\chi_{ij})$, $\Lambda=(\lambda_{ij})$ and $\Delta=(\delta_{ij})$.

If $M\in\mathcal{LC}_{g,1}$, then $m_{+,*}(A)\subseteq m_{-,*}(A)$  so $S$ is a basis for $H_1(M;\mathbb{Z})$ and all the coefficients $\chi_{ij}$ in equation (\ref{eq1lemma}) are zero. Thus $B$ is a homology cube and $X=0$. Conversely, if $B$ is a homology cube and $X=0$, then $M\in\mathcal{LC}_{g,1}$. Therefore we have $(i)$.

Now assuming that $B$ is a homology cube, we have that  $M\in\mathcal{ILC}$ if and only if $X=0$ and $\Lambda=\text{Id}_g$, thus we have ($ii$). Similarly, $M\in\mathcal{IC}$ if and only if $X=\Delta=0$ and $\Lambda=\text{Id}_{g}$, so we have ($iii$).
\end{proof}	

Using Lemma \ref{lemmalinking} let us now see that the inclusions $\mathcal{IC}\subseteq\mathcal{ILC}$ and $\mathcal{ILC}\subseteq\mathcal{LC}$ are strict.

\begin{example}\label{ejemplohopfs} Let $g\geq2$. Consider the identity cobordism $M=\Sigma_{g,1}\times[-1,1]$ and embed two framed Hopf links $L_1$ and $L_2$ as in Figure \ref{figura18}($a$). Perform  surgery along $L_1$ and $L_2$. The resulting cobordism $M_{L_1\cup L_2}$ belongs to $\mathcal{LC}$ but not to $\mathcal{ILC}$. This follows from Lemma \ref{lemmalinking}: in Figure \ref{figura18}($b$) we show the bottom-top tangle presentation of $M_{L_1\cup L_2}$, which allows to compute its linking matrix.
\begin{figure}[ht!]
										\centering
                        \includegraphics[scale=0.6]{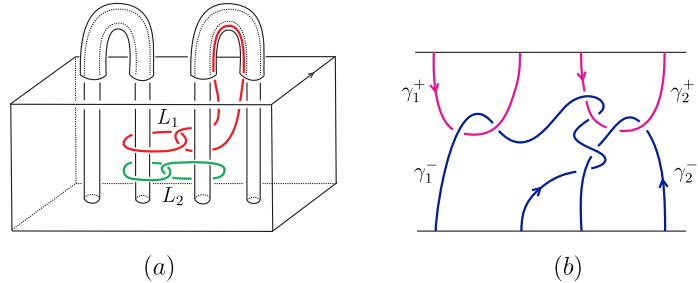}
												\caption{($a$) Embedding of $L_1$ and $L_2$  in $M$ and ($b$) bottom-top tangle presentation of $M_{L_1\cup L_2}$.}
												\label{figura18}
\end{figure}
\end{example}

\begin{example}\label{ejemplohopf} Let $g\geq2$. Consider  $M=\Sigma_{g,1}\times[-1,1]$ and embed a framed Hopf link $L$  as in Figure \ref{figura26}($a$).  The  resulting cobordism $M_{L}$, obtained after  surgery along $L$, belongs to $\mathcal{ILC}$ but not to $\mathcal{IC}$. In Figure \ref{figura26}($b$) we show the bottom-top tangle presentation of $M_{L}$, which allows to compute its linking matrix.
\begin{figure}[ht!]
										\centering
                        \includegraphics[scale=0.6]{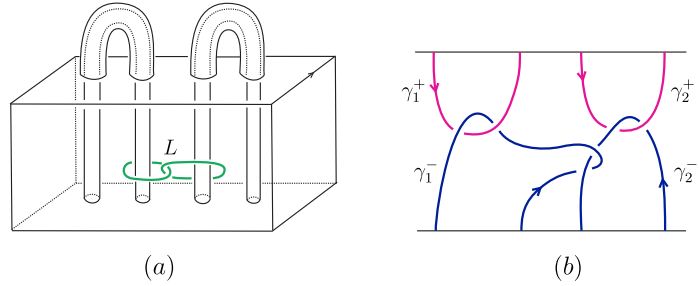}
												\caption{($a$) Embedding of $L$  in $M$ and ($b$) bottom-top tangle presentation of $M_L$.}
												\label{figura26}
\end{figure}
\end{example}

\begin{definition} The \emph{Lagrangian filtration} or \emph{Johnson-Levine filtration}  of $\mathcal{C}$ is the descending chain of submonoids $\{J_k^L\mathcal{C}\}_{k\geq1}$ defined as
$$J_k^L\mathcal{C}:=\{(M,m)\in \mathcal{ILC}\ \ |\ \ \forall\alpha\in\mathbb{A},\  \iota_{k+1}\rho_{k}(M)(\{\alpha\}_{k+1})=1\in\pi'/\Gamma_{k+1}\pi'\},$$
where $\iota_{k+1}:\pi/\Gamma_{k+1}\pi\rightarrow\pi'/\Gamma_{k+1}\pi'$ is  induced by $\iota_{\#}:\pi\rightarrow\pi'$.
\end{definition}
Notice that $J_1^L\mathcal{C}=\mathcal{ILC}$. To summarize, we have a descending chain of submonoids
\begin{equation}\label{ecuacion3.82}
\mathcal{C}\supseteq\mathcal{LC}\supseteq \mathcal{ILC}=J_1^L\mathcal{C}\supseteq J_2^L\mathcal{C}\supseteq\cdots
\end{equation} 

\begin{definition} Let $k\geq 1$. The $k$-th  \emph{Johnson-Levine homomorphism}
\begin{equation}\label{ecuacion3.22}
\tau_k^L:J_k^L\mathcal{C}\rightarrow\text{\emph{Hom}}(A,\Gamma_{k+1}\pi'/\Gamma_{k+2}\pi')\cong A^*\otimes \mathfrak{L}_{k+1}(H')\cong H'\otimes \mathfrak{L}_{k+1}(H'),
\end{equation}

\noindent is the map that sends $(M,m)\in J_k^L\mathcal{C}$ to the map   $a\in A\mapsto \{\iota_{k+2}\rho_{k+1}(M)(\{\alpha\})\}_{k+2}$, where  $\alpha\in\mathbb{A}$ is such that $\text{\emph{ab}}(\alpha)=a$.
\end{definition}

Notice that $\text{ker}(\tau_k^L)=J_{k+1}^L\mathcal{C}$. The same arguments as those used in the case of the mapping class group, work to show that $J_k^L\mathcal{C}$ is a submonoid of $\mathcal{C}$ and that $\tau_k^L$ is well defined, it is a monoid homomorphism, and that it takes values in $D_k(H')$.


\section{Properties of the Johnson-Levine homomorphisms}\label{seccion4}
In this section we study some properties of the Johnson-Levine homomorphisms and we compare the Johnson-Levine filtration to the  Johnson Levine filtration.


\subsection{Surjectivity of the  Johnson-Levine homomorphisms}

Let us start by the compatibility between the Johnson and Johnson-Levine homomorphisms. The surjective homomorphism $\iota_*:H\rightarrow H'$ induces surjective homomorphisms $\mathfrak{L}_{k}(H)\rightarrow \mathfrak{L}_{k}(H')$ which are compatible with the Lie bracket. 

\begin{lemma}\label{lema4.1} The  map $\iota_*:D_k(H)\rightarrow D_k(H')$ induced by $\iota_*:H\rightarrow H'$ is surjective.
\end{lemma}
\begin{proof} The result follows from the existence of a group section $s:H'\rightarrow H$ of the surjective homomorphism $\iota_*:H\rightarrow H'$ and the commutative diagram 
\begin{equation}\label{ecuacion4.1}
\xymatrix{  H\otimes\mathfrak{L}_{k+1}(H)\ar[r]^{\ \ \left[\ ,\ \right]}\ar[d]_{\iota_*\otimes\iota_*} & \mathfrak{L}_{k+2}(H) \ar[d]^{\iota_*} \\
						H'\otimes\mathfrak{L}_{k+1}(H')\ar[r]^{\ \ \left[\ ,\ \right]} & \mathfrak{L}_{k+2}(H').}
\end{equation}
More precisely, denote by $\Psi$ and $\Psi'$ the Lie brackets $\left[\ ,\ \right]:H\otimes\mathfrak{L}_{k+1}(H)\rightarrow\mathfrak{L}_{k+2}(H)$ and $\left[\ ,\ \right]:H'\otimes\mathfrak{L}_{k+1}(H')\rightarrow\mathfrak{L}_{k+2}(H')$, respectively. Let $y\in D_k(H')\subseteq H'\otimes\mathfrak{L}_{k+1}(H')$. The group section $s$ allows us to lift $y$ to $s(y)\in H\otimes\mathfrak{L}_{k+1}(H)$ such that $\iota_*(s(y))=y$. We deduce  $\Psi(s(y))=s\Psi'(y)=0$, hence $s(y)\in D_k(H)$.
\end{proof}

As pointed out by J. Levine in \cite[Section 4]{MR2265877}, in the mapping class group case,  the Johnson homomorphisms and the Johnson-Levine homomorphisms are compatible. This also holds for the monoid of homology cobordisms.

\begin{proposition}\label{proposition1seccion4} For every positive integer $k$, the  diagram 
\begin{equation}\label{ecuacion4.2}
\xymatrix{  J_k\mathcal{C}\ar[r]^{\subset}\ar[d]_{\tau_k} & J_k^L\mathcal{C} \ar[d]^{\tau_k^L} \\
						D_k(H)\ar[r]^{\iota_*} & D_k(H')}
\end{equation}
is commutative.
\end{proposition}

\begin{proof}
From the definitions, it is clear that   $J_k\mathcal{C}\subseteq J_k^L\mathcal{C}$ for all $k\geq1$. Consider the free basis $\{\alpha_1,\ldots,\alpha_g,\beta_1,\ldots,\beta_g\}$ of $\pi$  and the symplectic basis   $\{a_1,\ldots, a_g, b_1,\ldots,b_g\}$ of $H$ fixed in subsection \ref{seccion3.1}. Let $M=(M,m)\in J_k\mathcal{C}\subseteq J_k^L\mathcal{C}$. In these  bases, the $k$-th Johnson homomorphism is given by the formula
\begin{equation}\label{ecuacion1propo}
\tau_k(M)=\sum_{j=1}^g a_j\otimes\left(\rho_{k+1}(M)(\{\beta_j\})\cdot \{\beta_j^{-1}\}_{k+2}\right)-\sum_{j=1}^g b_j\otimes\left(\rho_{k+1}(M)(\{\alpha_j\}))\cdot \{\alpha_j^{-1}\}_{k+2}\right).
\end{equation}

\noindent Similarly, the $k$-th Johnson-Levine homomorphism is given by the formula
\begin{equation}\label{ecuacion2propo}
\tau_k^L(M)=-\sum_{j=1}^g \iota_*(b_j)\otimes\left(\iota_{k+2}\rho_{k+1}(M)(\{\alpha_j\})\right). 
\end{equation}
\noindent Thus by applying $\iota_*:D_k(H)\rightarrow D_k(H')$ to equation (\ref{ecuacion1propo}) we obtain $\iota_*\tau_k(M)=\tau_k^L(M)$.
\end{proof}

From Proposition \ref{proposition1seccion4} and Theorem \ref{thmgl} we obtain the following corollary.

\begin{corollary}\label{corollaryequ4.3} For every positive integer $k$, we have
\begin{equation}\label{ecuacion4.3}
 \text{\emph{ker}}(D_k(H)\stackrel{\iota_*}{\longrightarrow} D_k(H'))=\tau_k(J_k\mathcal{C}\cap J_{k+1}^L\mathcal{C}).
\end{equation}
\end{corollary}

According to  Lemma \ref{lema4.1}, Proposition \ref{proposition1seccion4} and Theorem \ref{thmgl}, we have the following.

\begin{corollary}[J. Levine \cite{MR1823501} Theorem 8]\label{proposicion4.1} For all $k\geq 1$, the Johnson-Levine homomorphism $\tau_k^L:J_k^L\mathcal{C}\rightarrow D_k (H')$ is surjective.
\end{corollary} 

The proof of J. Levine does not use Theorem \ref{thmgl}, instead, it uses the Oda embedding \cite{oda1992lower} and the surjectivity of Milnor invariants. More precisely, by choosing an  embedding of the disk $D_g^o$ with $g$ holes into $\Sigma$, we obtain the so-called \emph{Oda embedding} $\mathcal{O}:\mathcal{S}_g\longrightarrow \mathcal{C}_{g,1}$, see \cite[Section 3.2]{MR1823501}. This embedding relates the Milnor filtration with the Johnson filtration and it is compatible with the Milnor maps and the Johnson-Levine homomorphisms. These properties of the Oda embedding imply the surjectivity of the Johnson-Levine homomorphisms, for further details see \cite[Theorem 8]{MR1823501}.

\subsection{Invariance under the   \texorpdfstring{$Y_k$}{}-equivalence relation}\label{seccion4.2} In order to compare the Johnson filtration and the Johnson-Levine filtration, from our approach, we need to take some quotients of $J_k\mathcal{C}$ and $J_k^L\mathcal{C}$ by some equivalence relations to obtain a group structure compatible with the Johnson and Johnson-Levine homomorphisms. There are  at least two  ways to obtain a group from the monoid of homology cobordisms. One way is to consider homology cobordisms up to  4-dimensional homology bordism, see \cite{MR1823501,MR2131016}. Another way is to consider homology cobordisms up to $Y_k$-equivalence. We follow the latter approach. The notion of $Y_k$-equivalence was introduced independently by  M. Goussarov in \cite{MR1715131,MR1793618}  and by K. Habiro in \cite{MR1735632} in their study of finite type invariants. Here, we follow the terminology of  \cite{MR1735632}.

Let $G$ be a graph that can be decomposed into two subgraphs, say  $G=G'\cup G^o$, such that $G'$ is a unitrivalent graph and $G^o$ is a union of looped edges of $G$. The subgraph $G'$ is called the \emph{shape} of $G$. Let us consider a pair $(M,\gamma)$, where $M$ is a compact oriented $3$-manifold (possibly with boundary) and $\gamma$ is a framed oriented tangle (possibly empty)  in $M$ such that $\partial\gamma$ (if any) are fixed points in $\partial M$. A \emph{graph clasper} in $(M,\gamma)$ is an embedding $\mathbb{G}\hookrightarrow \text{int}(M\setminus\gamma)$ of a thickening $\mathbb{G}$ of $G$, see Figure \ref{figura4.9}. We still denote the image of the embedding by $G$. In particular, if the shape  of $G$ is simply connected, we call it a \emph{tree clasper}. The \emph{degree} of a  graph clasper is the number of trivalent vertices of its shape. If $G$ has degree 1 we call it a $Y$-\emph{clasper}. From now on, we assume that the degree of  graph claspers is greater than or equal to $1$.

\begin{figure}[ht!] 
										\centering
                        \includegraphics[scale=0.6]{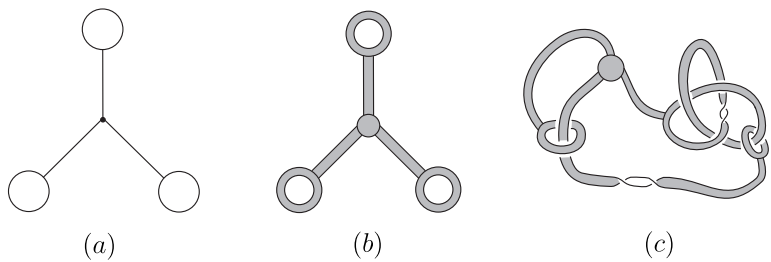}
												\caption{$(a)$ Graph $G$. $(b)$ Thickening $\mathbb{G}$. $(c)$ Embedding $\mathbb{G}\hookrightarrow M$ .}
												\label{figura4.9}
\end{figure}

A graph clasper $G$ in $(M,\gamma)$ carries surgery instructions for modifying this pair as follows. Suppose that $G$ has degree $1$. Consider a regular neighborhood $N(G)$ of $G$ in $\text{int}(M\setminus\gamma)$. Perform surgery in $N(G)$ along the framed six-component link $L$ illustrated in Figure \ref{figura4.10}. 

\begin{figure}[ht!] 
										\centering
                        \includegraphics[scale=0.48]{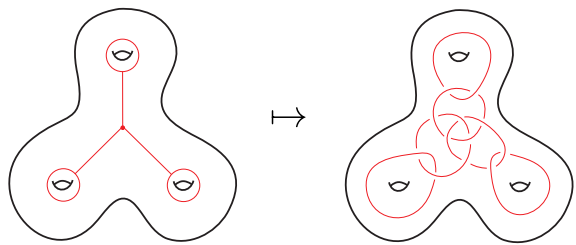}
												\caption{Framed link associated to a $Y$-clasper.}
												\label{figura4.10}
\end{figure}

\noindent Denote the result by $N(G)_L$. We obtain a new pair $(M_G,\gamma_G)$ by setting
$$M_G:=\left(M\setminus N(G)\right)\cup N(G)_L,$$
and $\gamma_G$ equal to the trace of $\gamma$ under the surgery. If $G$ is of degree $>1$ we apply the \emph{fission rule},  illustrated in Figure \ref{figura4.11}, until obtaining a disjoint union of $Y$-claspers. Then $(M_G,\gamma_G)$ is defined by performing surgery as before along each $Y$-clasper.  We say that $(M_G,\gamma_G)$ is obtained from $(M,\gamma)$ by a  $Y_k$-\emph{surgery}, where $k$ is the degree of $G$.

\begin{figure}[ht!] 
										\centering
                        \includegraphics[scale=0.7]{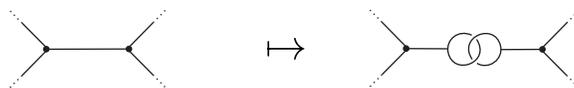}
												\caption{Fission rule.}
												\label{figura4.11}
\end{figure}

The $Y_k$-\emph{equivalence} is the equivalence relation among pairs $(M,\gamma)$ generated by $Y_k$-surgeries and orientation-preserving homeomorphisms. For $l\geq k$,  $Y_l$-equivalence implies $Y_k$-equivalence (this follows from Move $2$ and Move $9$ in \cite[Section 2.4]{MR1735632}).

Let us restrict to  the monoid of homology cobordisms.  K. Habiro proved in \cite[Theorem 5.4]{MR1735632} that $\mathcal{IC}/Y_r$ is a group. His proof is done in the setting of  string links but  the  arguments are the same  for $\mathcal{IC}$, see also \cite[Theorem 9.2]{MR1715131}.  From the short exact sequence
\begin{equation*}
1\longrightarrow \mathcal{IC}/{Y_r}\stackrel{\subset}{\longrightarrow} \mathcal{C}/{Y_r}\stackrel{\rho_1}{\longrightarrow}\text{Sp}(H)\longrightarrow 1,
\end{equation*}
it follows that $\mathcal{C}/Y_r$ is also a group. From  \cite[Lemma 6.1]{MR3074379} it follows that the  homomorphism $\rho_{k+1}:\mathcal{C}\rightarrow\text{Aut}(\pi/\Gamma_{k+2}\pi)$  is invariant under $Y_{k+1}$-equivalence. Therefore, the Johnson homomorphism $\tau_k$ and the Johnson-Levine homomorphism $\tau_k^L$ are invariant under $Y_{k+l}$-equivalence for all $l\geq 1$. 

\begin{lemma}For $r\geq k\geq 1$, the group $\mathcal{C}/Y_r$ contains $J_k\mathcal{C}/Y_r$ and $J_k^L\mathcal{C}/Y_r$ as subgroups.
\end{lemma}
\begin{proof}
It is enough  to show that $J_k\mathcal{C}/Y_r$ and $J_k^L\mathcal{C}/Y_r$  are closed under inverses. Let $\{M\}\in J_k\mathcal{C}/Y_r$, then there exists $\{N\}\in \mathcal{C}/Y_r$ such that $\{N\}\{M\}=\{\Sigma\times\left[-1,1\right]\}$ in $\mathcal{C}/Y_r$. By the invariance of $\rho_k$ under $Y_r$-equivalence, we have 
$$\text{Id}_{\pi/\Gamma_{k+1}\pi}=\rho_k(N)\circ\rho_k(M)=\rho_k(N),$$
hence $\{N\}\in J_k\mathcal{C}/Y_r$. 

Now, let us show by induction on $k$ that $J^L_k\mathcal{C}/Y_r$ is closed under inverses. Suppose that $k=1$ and let $\{M\}\in J^L_1\mathcal{C}/Y_r$. Consider $\{N\}\in \mathcal{C}/Y_r$ such that $\{N\}\{M\}=\{\Sigma\times\left[-1,1\right]\}$ in $\mathcal{C}/Y_r$. By the invariance of $\rho_1$ under $Y_r$-equivalence, we have $\text{Id}_{H}=\rho_1(N)\circ\rho_1(M)$. Let $a\in A$, thus $\rho_1(N)(a)=\rho_1(N)(\rho_1(M)(a))=a$. Therefore $\{N\}\in J^L_1\mathcal{C}/Y_r$. Next, suppose that $k\geq 2$ and let $\{M\}\in J^L_k\mathcal{C}/Y_r$. Since $J^L_k\mathcal{C}\subseteq  J^L_{k-1}\mathcal{C}$, by induction there exists  $\{N\}\in J^L_{k-1}\mathcal{C}/Y_r$ such that $\{N\}\{M\}=\{\Sigma\times\left[-1,1\right]\}$ in $J^L_{k-1}\mathcal{C}/Y_r$. On the other hand $\tau^L_{k-1}(M)=0$, so we have
$$\tau^L_{k-1}(N)=\tau^L_{k-1}(N)+\tau^L_{k-1}(M)=0.$$
Hence $\{N\}\in J^L_k\mathcal{C}/Y_r$. 
\end{proof}

\subsection{Comparison of  the Johnson and Johnson-Levine filtrations}

Consider the handlebody $V$ as in subsection \ref{seccion3.1}, seeing it as a cobordism from $\Sigma$ to $\Sigma_{0,1}=D$, the fixed disk on $\partial V$,  see Figure \ref{figure4.20}.
\begin{figure}[H] 
										\centering
                        \includegraphics[scale=0.7]{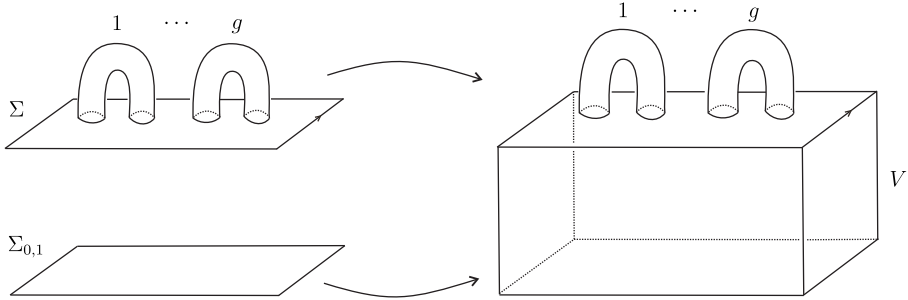}
												\caption{The handlebody  $V$ as a cobordism from $\Sigma$ to $\Sigma_{0,1}$.}
												\label{figure4.20}
\end{figure}

 Denote by $\mathcal{HC}$ the submonoid of $\mathcal{C}$ consisting of the cobordisms  $(M,m)$ such that $M\cup_{m_-}V$ is equal to $V$ as cobordisms. Notice that $\mathcal{HC}\cap \mathcal{M}$ is the handlebody group $\mathcal{H}$ defined in Remark \ref{handlebodygroup}.

\begin{lemma}\label{lemma4.10} We have the inclusion $\mathcal{HC}\cap \mathcal{ILC}\subseteq \bigcap_{k\geq1} J_k^L\mathcal{C}$.
\end{lemma}
\begin{proof}
Consider the system of meridians and parallels  $\{\alpha_1,\ldots,\alpha_g,\beta_1,\ldots,\beta_g\}$ of $\Sigma$ and denote in the same way an induced system of generators of $\pi$. Notice that $\mathbb{A}=\text{ker}(\pi_1(\Sigma)\stackrel{\iota_{\#}}{\longrightarrow}\pi_1(V))$ is the normal closure  of $\{\alpha_1,\ldots,\alpha_g\}$. Let $(M,m)\in\mathcal{HC}\cap\mathcal{ILC}$. It is enough to show that, for all $1\leq i\leq g$ and for all $k\geq 1$, we have $\rho_k(M)(\{\alpha_i\})\in (\mathbb{A}\cdot\Gamma_{k+1}\pi)/\Gamma_{k+1}\pi$. Indeed, since $M\cup_{m_-}V\cong V$, the curve $m_+(\alpha_i)$ bounds a disk $D^i$ in $M\cup_{m_-}V$. Now, some of the curves $m_-(\beta_j)$ intersect $D^i$ in a transversal way. Hence $m_{+,\#}(\alpha_i)$ can be written as a product of homotopy classes of  meridians associated to those curves $m_-(\beta_j)$ intersecting $D^i$. Since all the meridians are conjugates, we conclude that $m_{+,\#}(\alpha_i)$ can be written  as a product of conjugates of the homotopy classes of  the curves  $m_{-}(\alpha_j)$.  Hence $\rho_k(M)(\{\alpha_i\})$ belongs to $(\mathbb{A}\cdot\Gamma_{k+1}\pi)/\Gamma_{k+1}\pi$. Therefore $M$ belongs to  $\bigcap_{k\geq1} J_k^L\mathcal{C}$.
\end{proof}

The other inclusion does not hold. To see this, consider any homology sphere $P$ not homeomorphic to $\mathbb{S}^3$.   The connected sum  $M=(\Sigma\times [-1,1])\# P$ is a homology cobordism, which by  construction, belongs to $\bigcap_{k\geq1} J_k^L\mathcal{C}$ and does not belong to $\mathcal{HC}$. This contrasts with the mapping class group case where the respective equality holds, see Remark \ref{handlebodygroup}.

\begin{proposition}\label{propdif} For all $k\geq 1$, we have $\text{\emph{ker}}\left(D_k(H)\rightarrow D_k(H')\right)=\tau_k\left(\mathcal{HC}\cap J_k\mathcal{C}\right)$.
\end{proposition}

We postpone the proof of this proposition to subsection \ref{section4.4}.

\begin{lemma}\label{corolario2} For all $k,l\geq 1$, we have 
\begin{equation}\label{eq1cor}
\tau_k\left(\mathcal{HC}\cap J_k \mathcal{C}\right)=\tau_k\left(J_k\mathcal{C}\cap J_{k+1}^L\mathcal{C}\right),
\end{equation}
and 
\begin{equation}\label{eq2cor}
\frac{J_k\mathcal{C}\cap J_{k+1}^L\mathcal{C}}{Y_{k+1+l}}= \frac{J_{k+1}\mathcal{C}}{Y_{k+1+l}}\cdot q_{k+1+l}\left(\mathcal{HC}\cap J_k\mathcal{C}\right)
\end{equation}
in  $\mathcal{C}/Y_{k+1+l}$, where $q_{k+1+l}:\mathcal{C}\rightarrow\mathcal{C}/Y_{k+1+l}$ is the canonical projection.
\end{lemma}

\begin{proof}
Equality (\ref{eq1cor}) follows from Proposition \ref{propdif} and Corollary \ref{corollaryequ4.3}. Let us show  equality (\ref{eq2cor}). The inclusion ``$\supseteq$" follows from Lemma \ref{lemma4.10}. Let $M\in J_k\mathcal{C}\cap J_{k+1}^L\mathcal{C}$, thus by   (\ref{eq1cor}), $\tau_k(M)=\tau_k(U)$ for some element $U\in \mathcal{HC}\cap J_k\mathcal{C}$. But we can consider the inverse of $\{U\}$ in $\mathcal{C}/Y_{k+1+l}$, then $\{M\}\{U\}^{-1}\in \text{ker}(J_k\mathcal{C}/Y_{k+1+l}\xrightarrow{\tau_k} D_k(H))$. Thus $\{M\}=\{X\}\{U\}$ with $X\in\text{ker}(\tau_k)=J_{k+1}\mathcal{C}$.
\end{proof}

In \cite[Proposition 6.1]{MR2265877}, J. Levine showed that  $J_k^L\mathcal{M}=J_k\mathcal{M}\cdot\left(\mathcal{H}\cap\mathcal{IL}\right)$ for $k=1,2$ and he asked if this holds for all $k$. In the  case of homology cobordisms we have the following result.

\begin{theorem}\label{teorema1} For all $k,l\geq 1$,
\begin{equation}\label{equthm1}
\frac{J_k^L\mathcal{C}}{Y_{k+l}}=\frac{J_k\mathcal{C}}{Y_{k+l}}\cdot q_{k+l}\left(\mathcal{HC}\cap \mathcal{ILC}\right),
\end{equation}
where $q_{k+l}:\mathcal{C}\rightarrow\mathcal{C}/Y_{k+l}$ is the canonical projection.
\end{theorem}
\begin{proof}
By  Lemma \ref{lemma4.10}, $\left(J_k\mathcal{C}/Y_{k+l}\right)\cdot q_{k+l}\left(\mathcal{HC}\cap \mathcal{ILC}\right)$ is contained in $J_k^L\mathcal{C}/Y_{k+l}$  . Let us show the other inclusion by induction on $k$.
The argument for the case  $k=1$ is similar to the one used by J. Levine in \cite[Proposition 6.1]{MR2265877}. Indeed, let $M\in \mathcal{ILC}/Y_{1+l}$ with $\rho_1(M)\in\text{Sp}(H)$.  Identify $\text{Sp}(H)$ with $\text{Sp}(2g,\mathbb{Z})$ as in subsection \ref{seccion3.1} . Now, every matrix $\left( \begin{smallmatrix} \text{Id}_g & \Lambda\\ 0 & \text{Id}_g \end{smallmatrix} \right)$ in $\text{Sp}(H)$ can be realized as the image by $\rho_1$ of an element in $\mathcal{H}\cap\mathcal{IL}$, see  \cite[Lemma 6.3]{MR2265877}. Let $P\in \mathcal{H}\cap\mathcal{IL}$  that  realizes the matrix  $\rho_1(M)$ and  consider the inverse $\{P\}^{-1}$ of $\{P\}$ in $\mathcal{C}/Y_{1+l}$ (this is actually the class of  the inverse of $P$  in $\mathcal{M}$). Hence $\{M\}\{P\}^{-1}$ acts trivially in homology, that is, $\{M\}\{P\}^{-1}=\{N\}\in\mathcal{IC}=J_1\mathcal{C}$. Therefore 
$$\{M\}=\{N\}\{P\}\in \frac{J_1\mathcal{C}}{Y_{1+l}}\cdot q_{1+l}\left(\mathcal{HC}\cap \mathcal{ILC}\right).$$ 

\noindent Suppose that the inclusion ``$\subseteq$" in (\ref{equthm1})  is true for $k$. Thus we have 
\begin{equation}\label{ecuacion4.11}
\frac{J_{k+1}^L\mathcal{C}}{Y_{k+1+l}}\subseteq \frac{J_k^L\mathcal{C}}{Y_{k+1+l}}\subseteq \frac{J_k\mathcal{C}}{Y_{k+1+l}}\cdot q_{k+1+l}\left(\mathcal{HC}\cap \mathcal{ILC}\right).
\end{equation}
Let $M\in J_{k+1}^L\mathcal{C}$. By  the above inclusion we can write $\{M\}=\{N\}\{P\}$ with $N\in J_k\mathcal{C}$ and $P\in \mathcal{HC}\cap\mathcal{ILC}$. Notice that $\tau_k^L(P)=0$ by Lemma \ref{lemma4.10}. Since $\tau_k^L$ is invariant under $Y_{k+1+l}$-surgery (see  subsection \ref{seccion4.2}), we have 
 $$0=\tau_k^L(M)=\tau_k^L(N)+\tau_k^L(P)=\tau_k^L(N),$$
 therefore $N\in \text{ker}(\tau_k^L)=J_{k+1}^L\mathcal{C}$. Hence $N\in J_k\mathcal{C}\cap J_{k+1}^L\mathcal{C}$.

\noindent From equality (\ref{eq2cor}) in Lemma \ref{corolario2}, it follows  that $\{N\}\in \left(J_{k+1}\mathcal{C}/Y_{k+1+l}\right)\cdot q_{k+1+l}\left(\mathcal{HC}\cap J_k\mathcal{C}\right)$. Hence
\begin{align*}
\{M\}=\{N\}\{P\}&\in \left(\frac{J_{k+1}\mathcal{C}}{Y_{k+1+l}}\cdot q_{k+1+l}\left(\mathcal{HC}\cap J_k\mathcal{C}\right)\right)\cdot q_{k+1+l}\left(\mathcal{HC\cap\mathcal{ILC}}\right)\\
								 &\subseteq \frac{J_{k+1}\mathcal{C}}{Y_{k+1+l}}\cdot q_{k+1+l}\left(\mathcal{HC}\cap \mathcal{ILC}\right),
\end{align*}
which completes the proof.
\end{proof}

\subsection{Proof of Proposition \ref{propdif}}\label{section4.4}
Let us start by reviewing some preliminaries, we follow  \cite{MR2240921}. Consider the \emph{free quasi-Lie algebra} 
$$\mathfrak{L}^q(H)=\bigoplus_{k\geq 1}\mathfrak{L}^q_{k}(H)$$
generated by the $\mathbb{Z}$-module $H$, that is, instead of the relation $[x,x]=0$ in $\mathfrak{L}(H)$ we have the antisymmetry relation $[x,y]+[y,x]=0$.
Let $D^q_k(H)$ denote the kernel of the quasi-Lie bracket map $\left[\ ,\ \right]: H\otimes\mathfrak{L}^q_{k+1}(H)\rightarrow\mathfrak{L}^q_{k+2}(H)$. There is a canonical map  $\mathfrak{L}^q(H)\rightarrow\mathfrak{L}(H)$, which induces a homomorphism $D^q(H)=\bigoplus_{k\geq1} D^q_k(H)\rightarrow D(H)=\bigoplus_{k\geq1}D_k(H)$.

We can define a homomorphism
\begin{equation}\label{equ1prop4.7}
\eta^q_k:\mathcal{T}_k(H)\longrightarrow D_k^q(H) 
\end{equation}
in the same way that we defined the homomorphism $\eta^{\mathbb{Z}}_k:\mathcal{T}_k(H)\rightarrow D_k(H)$ in subsection \ref{subsection2.5}:  the composition of $\eta^q_k$ with the canonical map $D^q_k(H)\rightarrow D_k(H)$ is exactly $\eta^{\mathbb{Z}}_k$. Recall that we denote by $\eta_k:\mathcal{T}_k(H)\otimes\mathbb{Q}\rightarrow D_k(H)\otimes\mathbb{Q}$ the rationalization of $\eta^{\mathbb{Z}}_k$.  J. Levine carried in \cite{MR2240921} a detailed study of the homomorphism $\eta^q_k$. In particular, he obtained (\cite[Corollary 2.3]{MR2240921}) for  all $j\geq 1$ the following short exact sequences
\begin{equation}\label{equ2prop4.7}
0\longrightarrow H\otimes\mathfrak{L}_j(H)\otimes \mathbb{Z}/2\stackrel{s}{\longrightarrow} D^q_{2j-1}(H)\longrightarrow D_{2j-1}(H)\longrightarrow 0
\end{equation}
where $s(h\otimes u\otimes 1)=h\otimes [u,u]$ for $h\in H$ and $u\in \mathfrak{L}_j(H)$, and
\begin{equation}\label{equ3prop4.7}
0\longrightarrow D^q_{2j}(H)\longrightarrow D_{2j}(H)\stackrel{p}{\longrightarrow}\mathfrak{L}_{j+1}(H)\otimes\mathbb{Z}/2\longrightarrow 0.
\end{equation}
To describe the map $p$ in (\ref{equ3prop4.7}) let us first recall from \cite[Remark 2.4]{MR2240921} some elements of $D_{2j}(H)$ which do not come from $D^q_{2j}(H)$. Let $u\in\mathfrak{L}_{j+1}(H)$ and denote by $\text{tr}(u)$ the associated rooted tree. Let $\text{tr}(u)\odot\text{tr}(u)$ be the  Jacobi diagram obtained by joining the roots of two copies of $\text{tr}(u)$. The element $\eta_{2j}(\frac{1}{2}\text{tr}(u)\odot\text{tr}(u))$ belongs to $D_{2j}(H)$ and does not belong to $D^{q}_{2j}(H)$. The map $p$ sends $\eta_{2j}(\frac{1}{2}\text{tr}(u)\odot\text{tr}(u))$  to $u\otimes 1$.

J. Levine also  proved  \cite[Theorem 1]{MR1943338} that the map $\eta^q_k$ is surjective and that $(k+2)\text{ker}(\eta^q_k)=0$. (These results imply, in particular, that $\eta_k$ is an isomorphism, as we recalled at the end of subsection \ref{subsection2.5}). From  the exact sequences (\ref{equ2prop4.7}) and (\ref{equ3prop4.7}) together with the surjectivity of $\eta^q_k$ we deduce the following.

\begin{corollary}\label{corollary1prop4.7} For all $j\geq 1$, 
\begin{enumerate}
\item[\emph{(\emph{i})}] $D_{2j-1}(H)$ is generated by the elements $\eta^{\mathbb{Z}}_{2j-1}(v)$ with $v\in \mathcal{T}_{2j-1}(H)$.
\item[\emph{(\emph{ii})}] $D_{2j}(H)$ is generated by the elements $\eta^{\mathbb{Z}}_{2j}(v)$ with $v\in \mathcal{T}_{2j}(H)$ and
$\eta_{2j}(\frac{1}{2}\text{\emph{tr}}(u)\odot\text{\emph{tr}}(u))$ with $u\in\mathfrak{L}_{j+1}(H)$.
\end{enumerate}
\end{corollary}

\begin{lemma}\label{lemma1prop4.7}
Let $S=\{a_1,\ldots,a_g,b_1,\ldots, b_g\}$ be the fixed symplectic basis of $H$. For all $j\geq 1$,
\begin{enumerate}
\item[\emph{(\emph{i})}] $\text{\emph{ker}}(D_{2j-1}(H)\rightarrow D_{2j-1}(H'))$ is generated by the elements $\eta^{\mathbb{Z}}_{2j-1}(v)$ with $v$ a tree-like Jacobi diagram with legs colored by $S$ and at least one leg colored by some $a_i$.
\item[\emph{(\emph{ii})}] $\text{\emph{ker}}\left(D_{2j}(H)\rightarrow D_{2j}(H')\right)$ is generated by the elements $\eta^{\mathbb{Z}}_{2j}(v)$ with $v$ a tree-like Jacobi diagram with legs colored by $S$, with at least one leg colored by some $a_i$; and the elements $\eta_{2j}(\frac{1}{2}\text{\emph{tr}}(u)\odot\text{\emph{tr}}(u))$ with $u\in\mathfrak{L}_{j+1}(H)$ a Lie commutator which  has at least one $a_i$ as one of its components.
\end{enumerate}
\end{lemma}
\begin{proof} Let $k\geq 1$ and let $x\in\text{ker}(D_k(H)\rightarrow D_k(H'))$. By Corollary \ref{corollary1prop4.7}, we have
\begin{equation}\label{equ6prop4.7}
x=\sum_i\eta^{\mathbb{Z}}_k(v_i)+\sum_l\eta_k\left(\frac{1}{2}\text{tr}(u_l)\odot \text{tr}(u_l)\right)
\end{equation}
with $v_i\in\mathcal{T}_k(H)$, and $u_l\in\mathfrak{L}_{j+1}(H)$ if $k=2j$. Notice that if $k$ is odd, the second sum in equation (\ref{equ6prop4.7}) does not appear. By the linearity relation  we can suppose that all the $v_i$'s have legs colored by $S$ and that all the $u_l$'s are Lie commutators on $S$.  Let $y=\sum_i v_i+\sum_l\frac{1}{2}\text{tr}(u_l)\odot\text{tr}(u_l)\in\mathcal{T}_k(H)\otimes\mathbb{Q}$ and consider the commutative diagram
\begin{equation}\label{equ7prop4.7}
\xymatrix{  \mathcal{T}_k(H)\otimes\mathbb{Q}\ar[r]^-{\eta_k}_-{\cong}\ar[d]_-{\varphi} &D_k(H)\otimes\mathbb{Q} \ar[d]^-{\varphi'}\\
					\mathcal{T}_k(H')\otimes\mathbb{Q}\ar[r]_-{\eta_k}^-{\cong}&D_k(H')\otimes\mathbb{Q}}
\end{equation}
where $\varphi$ and $\varphi'$ are induced by the homomorphism $\iota_*:H\rightarrow H'$. We have that $\varphi'\eta_k(y)=0$, so $\eta_k\varphi(y)=0$. Since $\eta_k$ is an isomorphism, $\varphi(y)=0$. Let us write
 $$y=\sum_i v_i+\sum_l\frac{1}{2}\text{tr}(u_l)\odot\text{tr}(u_l)=y'+y'',$$
such that  all the diagrams appearing in $y'$ have at least one leg colored by some $a_i$, and   all the diagrams appearing in $y''$ have legs colored only by $\{b_1,\ldots, b_g\}$. Hence
$$0=\varphi(y)=\varphi(y')+\varphi(y'')=\varphi(y'').$$
Now, $\varphi(y'')=y''$ because all terms of $y''$ only have legs colored by $\{b_1,\ldots,b_g\}$. Thus $y''=0$, so $y=y'$. In other words, the diagrams appearing in $y$ whose legs are colored only by $\{b_1,\ldots,b_g\}$ can be grouped and they cancel out by  IHX and antisymmetry relations.
\end{proof}

Let us now turn  to the  proof of Proposition \ref{propdif}. Recall that we want to show that 
 $$\text{ker}(D_k(H)\stackrel{\iota_*}{\rightarrow} D_k(H'))=\tau_k(\mathcal{HC}\cap J_k\mathcal{C}).$$
 Let us first see the inclusion  ``$\supseteq$". Since $\mathcal{HC}\cap \mathcal{ILC}\subseteq \bigcap J_k^L\mathcal{C}$ (Lemma \ref{lemma4.10}), we have  that  $\tau_k^L(\mathcal{HC}\cap\mathcal{ILC})=0$ for all $k\geq 1$. Therefore, if $M\in\mathcal{HC}\cap J_k\mathcal{C}\subseteq \mathcal{HC}\cap\mathcal{ILC}$, then $\tau_k^L(M)=0$, so by Proposition \ref{proposition1seccion4} we have that $\iota_*(\tau_k(M))=\tau_k^L(M)=0$, that is, $\tau_k(M)\in\text{ker}(D_k(H)\stackrel{\iota_*}{\rightarrow} D_k(H'))$.

\noindent  We now show the inclusion ``$\subseteq$". According to Lemma \ref{lemma1prop4.7}, it is enough to prove for all $j\geq 1$ that
\begin{enumerate}
\item[(\emph{i})] $\eta^{\mathbb{Z}}_{2j-1}(v)\in\tau_{2j-1}(\mathcal{HC}\cap J_{2j-1}\mathcal{C})$ for $v$ as in Lemma \ref{lemma1prop4.7}($i$), and
\item[(\emph{ii})] $\eta^{\mathbb{Z}}_{2j}(v)\in\tau_{2j}(\mathcal{HC}\cap J_{2j}\mathcal{C})$ and $\eta_{2j}(\frac{1}{2}\text{tr}(u)\odot \text{tr}(u))\in\tau_{2j}(\mathcal{HC}\cap J_{2j}\mathcal{C})$ for $v$ and $u$ as in Lemma \ref{lemma1prop4.7}($ii$).
\end{enumerate}
Let $\mathcal{S}^{\text{odd}}_{2g}$ be the submonoid of string links $\sigma$ on $2g$ strands  in homology cubes,   with trivial linking matrix and satisfying the property that if we forget all the  odd components of $\sigma$, the obtained string link is  trivial. The Milnor-Johnson correspondence, described in subsection \ref{MJcorrespondence}, sends $\mathcal{HC}\cap\mathcal{IC}$ to $\mathcal{S}^{\text{odd}}_{2g}$. Hence by diagram (\ref{ecuacion4.9}),  proving ($i$) and $(ii)$ above is equivalent to show
\begin{enumerate}
\item[(\emph{iii})] $v\in\eta_k^{-1}\mu_{k+1}(\mathcal{S}^{odd}_{2g}\cap \mathcal{S}_{2g}[k+1])$ for  $k$ odd and  $v$ as in Lemma \ref{lemma1prop4.7}($i$), and 
\item[(\emph{iv})] $v,\frac{1}{2}\text{tr}(u)\odot \text{tr}(u)\in\eta_k^{-1}\mu_{k+1}(\mathcal{S}^{odd}_{2g}\cap \mathcal{S}_{2g}[k+1])$ for $k$ even and $v$ and $u$ as in Lemma \ref{lemma1prop4.7}($ii$).
\end{enumerate}
This can be done by using  a string link version of Cochran's realization theorems for Milnor invariants \cite[Theorem 7.2]{MR1042041} and \cite[Theorem 3.3]{MR1055569}: here we develop  \cite[Remark 8.2]{MR1783857}. This process is called \emph{antidifferentiation} and it is very close to surgery along tree claspers, see \cite[Section 7]{MR1735632}.  We sketch this process below.

 Let $S$ be as in Lemma \ref{lemma1prop4.7}. Suppose that $k=2j$. Consider $u\in\mathfrak{L}_{j+1}(H)$ a Lie commutator which  has at least one $a_i$ as one of its components. From the rooted tree $\text{tr}(u)$ we are going to recursively construct  a string link $L(\frac{1}{2}u)$ which realizes the diagram $\frac{1}{2}\text{tr}(u)\odot \text{tr}(u)$.

\textbf{Starting step}. Suppose that $u=[u_1,u_2]$. Consider the oriented  Whitehead link  and label its components by $u_1$ and $u_2$ respectively, see Figure \ref{figura21}($a$).

\textbf{Recursive step}. Suppose for example that $u_1\in\mathfrak{L}_{\geq 2}(H)$, say $u_1=[u_{11},u_{12}]$. Perform a $0$-twisted Bing doubling to the component labeled by $u_1$ and label the two new components $u_{11}$ and $u_{12}$ respectively. See Figure \ref{figura21}($b$).
\begin{figure}[ht!]
										\centering
                        \includegraphics[scale=0.79]{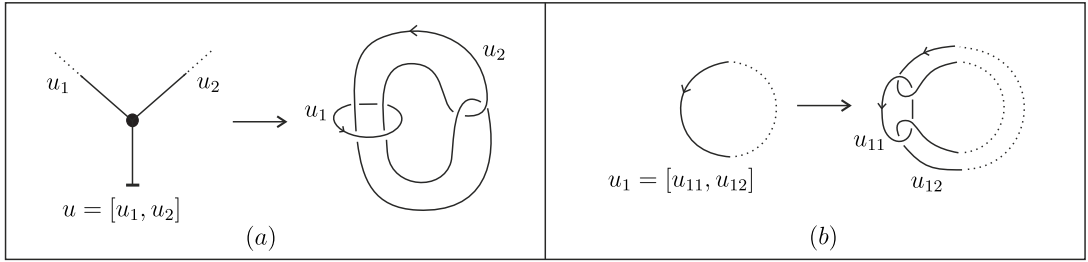}
												\caption{($a$) Starting step and ($b$) recursive step.}
												\label{figura21}
\end{figure}

\textbf{Banding step}. After finishing the above process we obtain a $(j+1)$-component link  whose components are labeled with elements of  $S$. Now, if two components have the same label, we perform an \emph{interior band sum} between the two components, see \cite[Section 7]{MR1042041} for more details. If necessary, add trivial components to the resulting link in order to obtain a $2g$-component link with components, each one, labeled by a unique element of $S$. Denote this link by $l(\frac{1}{2}u)$. Since $u$ is a Lie commutator with at least one $a_i$ as one of its components, the construction  implies that the  link $l(\frac{1}{2}u)$ becomes the trivial $g$-component  link if we forget all the components with labels $a_1,\ldots,a_g$.

\textbf{Final step}. Open the link $l(\frac{1}{2}u)$ to obtain a string link $l'(\frac{1}{2}u)$ on $2g$-strands, each one labeled by a unique element of $S$, satisfying  the property that if we forget all the components with labels $a_1,\ldots,a_g$ then it becomes the trivial $g$-component   string link. Now, by conjugating with the generators $\sigma_1,\ldots,\sigma_{2g-1}$ of the braid group on $2g$-strands, we  arrange the components of $l'(\frac{1}{2}u)$ in a such way  that the  $(2i)$-th component is labeled by $b_i$ and the $(2i-1)$-st component is labeled by $a_i$, for $i=1,\ldots,g$. Denote the resulting string link by $L(\frac{1}{2}u)$. We have that $L(\frac{1}{2}u)\in\mathcal{S}^{\text{odd}}_{2g}\cap\mathcal{S}_{2g}[k+1]$ and $\mu_{k+1}(L(\frac{1}{2}u))=\mu_{k+1}(l'(\frac{1}{2}u))=\pm\eta_k(\frac{1}{2}\text{tr}(u)\odot \text{tr}(u))$, see \cite[Remark 8.2]{MR1783857}, the sign depending on the clasp of the Whitehead link in the starting step.

\begin{example}
Let us illustrate the above process in a particular case. Suppose that $g=2$ and $u=[[a_1,b_1],a_1]$. We show in Figure \ref{figura22}($i$) the starting step, in  Figure \ref{figura22}($ii$) the recursive step. In  Figure \ref{figura22}($iii$) we perform an interior band sum and add trivial components. Finally in Figure \ref{figura23} we show the associated string link and the arrangement of its components.
\begin{figure}[ht!]
										\centering
                        \includegraphics[scale=0.79]{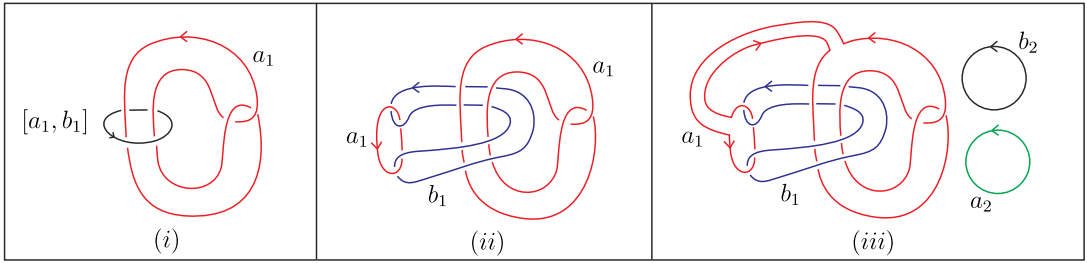}
												\caption{($i$) Starting step, ($ii$) recursive step and ($iii$)  link $l(\frac{1}{2}u)$ for $u=[[a_1,b_1],a_1]$.}
												\label{figura22}
\end{figure}

\begin{figure}[ht!]
										\centering
                        \includegraphics[scale=0.79]{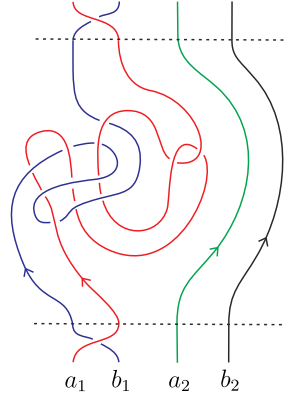}
												\caption{String link $L(\frac{1}{2}u)$ for $u=[[a_1,b_1],a_1]$.}
												\label{figura23}
\end{figure}

\end{example}

Now if $v$ is a tree-like Jacobi diagram as in Lemma \ref{lemma1prop4.7}, of i-deg $\geq 2$ (the case i-deg $=1$ is realized by a string link version of the Borromean rings), then chose any \emph{internal edge} of $v$ (edge connecting two trivalent vertices) and cut it in half to obtain  two rooted trivalent trees. Let $u_1$ and $u_2$ be the Lie commutators associated to these rooted trees. Notice that $v=\text{tr}(u_1)\odot \text{tr}(u_2)$. Consider the oriented Hopf link and label its components by $u_1$ and $u_2$ respectively, see Figure \ref{figura24}.

\begin{figure}[ht!]
										\centering
                        \includegraphics[scale=0.85]{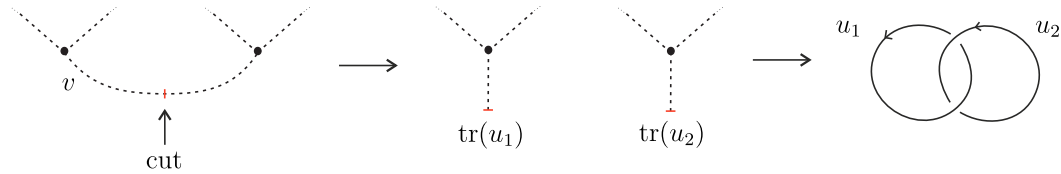}
												\caption{Starting the antidifferentiation process to realize $v$.}
												\label{figura24}
\end{figure}

\noindent Then continue the antidifferentiation process by performing the recursive step, banding step and final step. At the end we obtain a string link $L(v)\in\mathcal{S}^{\text{odd}}_{2g}\cap\mathcal{S}_{2g}[k+1]$ such that $\mu_{k+1}(L(v))=\pm\eta^{\mathbb{Z}}_k(v)$, the sign depending on the clasp of the  Hopf link that we started with to realize $v$. 
\section{The LMO functor and the Johnson-Levine homomorphisms}\label{seccion5}
 This section is devoted to the relation between the Johnson-Levine homomorphisms and the LMO functor. We refer to \cite{MR1881401,MR1931167,MR1931168} for an introduction to the LMO invariant and to \cite{MR2403806} for its functorial extension.

\subsection{Jacobi diagrams}\label{section5.1} In subsection \ref{subsection2.5} we reviewed the notion of tree-like Jacobi diagram. In this subsection we consider more general Jacobi diagrams.

A  \emph{Jacobi diagram} is a finite unitrivalent graph such that the trivalent vertices are \emph{oriented}, that is, its incident edges  are endowed with a cyclic order. Let $C$ be a finite set. We call a Jacobi diagram  $C$\emph{-colored} if its univalent vertices (or \emph{legs})  are colored with elements of the $\mathbb{Q}$-vector space spanned by $C$. The \emph{internal degree} of a Jacobi diagram is the number of trivalent vertices, we denote it by i-deg. The connected Jacobi diagram of i-deg $=0$ is called a \emph{strut}. As for tree-like Jacobi diagrams, we use dashed lines to represent Jacobi diagrams and, when we draw them, we assume that the orientation of trivalent vertices is counterclockwise. See Figure \ref{figura4.12} for some examples. 
\begin{figure}[ht!] 
										\centering
                        \includegraphics[scale=0.72]{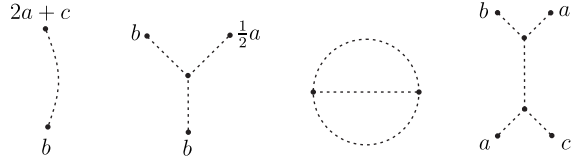}
												\caption{$C$-colored Jacobi diagrams of i-deg  $0, 1, 2$ and $2$, respectively. Here $C=\{a,b,c\}$ }
												\label{figura4.12}
\end{figure}

The space of $C$-colored Jacobi diagrams is defined as
$$\mathcal{A}(C):=\frac{\text{Vect}_{\mathbb{Q}}\{C\text{-colored Jacobi diagrams}\}}{\text{AS, IHX, $\mathbb{Q}$-multilinearity}},$$
\noindent where the relations AS, IHX  are local and the multilinearity relation applies to the $C$-colored legs, see Figure \ref{figura4.13} in subsection \ref{subsection2.5}.

The vector space $\mathcal{A}(C)$ is graded by the internal degree, thus we can consider the degree completion which we still denote by $\mathcal{A}(C)$, in other words, we also consider formal series of Jacobi diagrams. There is a product in $\mathcal{A}(C)$ given by disjoint union, and a coproduct defined by $\Delta(D):=\sum D'\otimes D''$ where the sum ranges over pairs of  subdiagrams $D',D''$ of $D$ such that $D'\sqcup D''=D$. With these structures, $\mathcal{A}(C)$ is a complete Hopf algebra. Its primitive part is the subspace $\mathcal{A}^c(C)$ spanned by connected Jacobi diagrams. We denote by $\mathcal{A}^Y(C)$ the subspace of Jacobi diagrams  such that all of their  connected components  have at least one trivalent vertex.  A Jacobi diagram in $\mathcal{A}(C)$ is \emph{looped} if it has a non-contractible component, for instance the third diagram in Figure \ref{figura4.12} is looped. The subspace generated by looped diagrams is an ideal. We denote by $\mathcal{A}^{Y,t}(C)$ the quotient of $\mathcal{A}^Y(C)$ by this ideal.

For $k\geq 1$ denote by $\mathcal{A}_k^{Y,t,c}(C)$ the subspace of $\mathcal{A}^{Y,t}(C)$ generated by connected  diagrams of i-deg $=k$. If $G$ is a finitely generated free abelian group, we define the space $\mathcal{A}(G)$  of $G$-colored Jacobi diagrams by $\mathcal{A}(G)=\mathcal{A}(C)$ where $C$ is any set of free generators of $G$.  In particular for the abelian group $H=H_1(\Sigma_{g,1};\mathbb{Z})$   we have 
$$\mathcal{A}_k^{Y,t,c}(H)=\mathcal{T}_k(H)\otimes \mathbb{Q},$$
where $\mathcal{T}(H)=\bigoplus _{k\geq 1} \mathcal{T}_k(H)$ is the group of tree-like Jacobi diagrams defined in subsection \ref{subsection2.5}.


\subsection{The LMO functor}\label{section5.3} Let us start by the definition of the target category ${}^{ts}\!\!\mathcal{A}$ of the LMO functor. For a non-negative integer $g$,  denote by $\left\lfloor g \right\rceil^*$ the set $\{1^*,\ldots, g^* \}$, where $*$ is a symbol like $+$, $-$ or $*$ itself.   The objects of the category ${}^{ts}\!\!\mathcal{A}$ are non-negative integers. The set of morphisms from $g$ to $f$ is the subspace ${}^{ts}\!\!\mathcal{A}(g,f)$ of diagrams in $\mathcal{A}(\left\lfloor g\right\rceil^+\sqcup \left\lfloor f\right\rceil^-)$ without struts whose both ends are colored by elements of $\left\lfloor g\right\rceil^+$. If $D\in {}^{ts}\!\!\mathcal{A}(g,f)$ and $E\in {}^{ts}\!\!\mathcal{A}(h,g)$ the composition $D\circ E$ is the element in ${}^{ts}\!\!\mathcal{A}(h,f)$ given by the sum  of Jacobi diagrams obtained by considering all the possible ways of gluing the $\left\lfloor g\right\rceil^+$-colored legs of $D$ with the $\left\lfloor g\right\rceil^-$-colored legs of $E$. Schematically

\centerline{\includegraphics[scale=0.85]{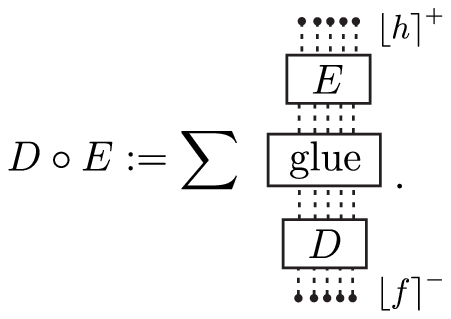}}

\noindent For example,

\centerline{\includegraphics[scale=0.9]{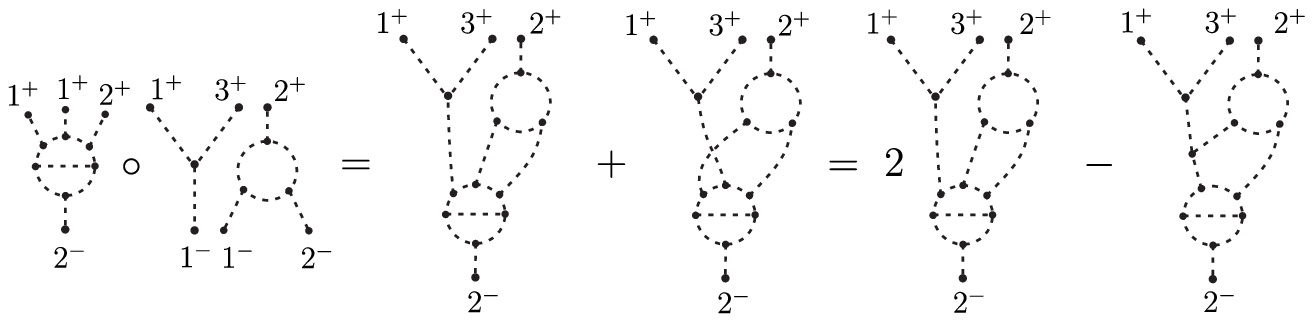}}

\medskip

\noindent where the last equality follows from the IHX relation. The identity morphism in ${}^{ts}\!\!\mathcal{A}(g,g)$ is given by

\centerline{\includegraphics[scale=0.85]{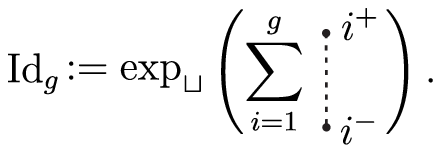}}

\noindent The category ${}^{ts}\!\!\mathcal{A}$ is called the category of \emph{top-substantial Jacobi diagrams}. 

Now, let us define the source category $\mathcal{LC}ob$ of the LMO functor, which is called the category of \emph{Lagrangian} cobordisms. The objects of  $\mathcal{LC}ob$ are non-negative integers. For all $g\geq1$, let us fix the handlebody $V_g$ and the inclusion $\iota:\Sigma_{g,1}\hookrightarrow V_g$ as in subsection \ref{seccion3.1}. A cobordism $(M,m)$ over $\Sigma_{g,1}$ belongs to $\mathcal{LC}ob(g,g)$ if it satisfies $H_1(M)=m_{-,*}(A_g)+m_{+,*}(H_1(\Sigma_{g,1};\mathbb{Z}))$ and  $m_{+,*}(A_g)\subseteq m_{-,*}(A_g)$. Recall that $A_g$ denotes the kernel of $H_1(\Sigma_{g,1};\mathbb{Z})\stackrel{\iota_*}{\longrightarrow}H_1(V_g;\mathbb{Z})$. In particular we have that the monoid of Lagrangian homology cobordisms $\mathcal{LC}_{g,1}$ is contained in $\mathcal{LC}ob(g,g)$. More generally, the set $\mathcal{LC}ob(g,f)$ is defined in a similar way  by considering cobordisms from $\Sigma_{g,1}$ to $\Sigma_{f,1}$.

For the definition of the LMO functor we need to use the \emph{Kontsevich integral}, because of this, it is necessary to change the objects of $\mathcal{LC}ob$ to obtain the category $\mathcal{LC}ob_q$: instead of non-negative integers, the objects of  $\mathcal{LC}ob_q$ are non-associative words in the single letter $\bullet$. We refer to \cite{MR2403806} for  more details.

Roughly speaking, the LMO functor $\widetilde{Z}:\mathcal{LC}ob_q\rightarrow {}^{ts}\!\!\mathcal{A}$ is defined as follows. Let $M$ be a Lagrangian cobordism (for example $M\in\mathcal{LC}_{g,1}$) and consider its  bottom-top tangle presentation $(B,\gamma')$. Next, take a \emph{surgery presentation} of $(B,\gamma')$, that is, a framed link $L\subseteq \text{int}([-1,1]^3)$ and a bottom-top tangle $\gamma$ in $[-1,1]^3$ such that surgery along $L$ carries $([-1,1]^3,\gamma)$ to $(B,\gamma')$. Then take the \emph{Kontsevich integral} of the pair $([-1,1]^3,L\cup\gamma)$, which gives a series of a kind of Jacobi diagrams. To get rid of the ambiguity in the surgery presentation, it is necessary to use some combinatorial operations on the space of diagrams. Among these operations, the so-called \emph{Aarhus integral} (see \cite{MR1931167,MR1931168}), which is a kind of formal Gaussian integration on the space of diagrams. In this way we arrive  to ${}^{ts}\!\!\mathcal{A}$. Finally, to obtain the functoriality, it is necessary to do a normalization. 

We emphasize that the definition of the Kontsevich integral requires the choice of a \emph{Drinfeld associator}, and the bottom-top tangle presentation requires the choice of a system of meridians and parallels.  Thus the LMO functor also depends on these choices.

\begin{example}\label{examplepsi} In \cite[Section 5.3]{MR2403806} the value of the LMO functor was calculated  in low degrees for the generators of $\mathcal{LC}ob$, when the chosen Drinfeld associator is even. For instance, for the Lagrangian cobordisms $\psi_{1,1}$ with bottom-top tangle presentation given in Figure \ref{figura4.28}, we have 

\medskip

\centerline{\includegraphics[scale=0.83]{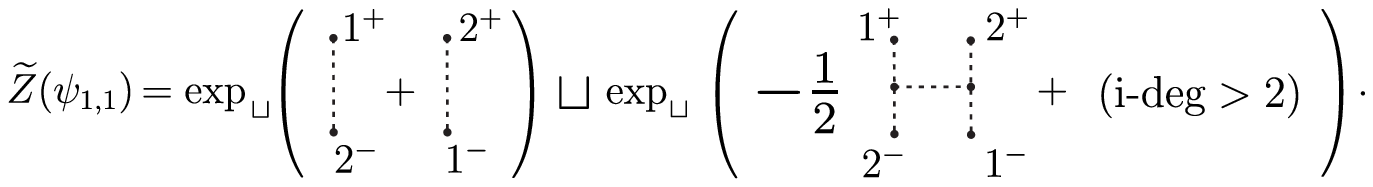}}

\medskip

\begin{figure}[ht!] 
										\centering
                        \includegraphics[scale=0.64]{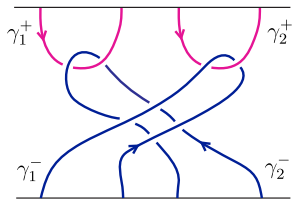}
												\caption{Bottom-top tangle presentation of $\psi_{1,1}$.}
												\label{figura4.28}
\end{figure}
\end{example}

For a matrix $\Lambda=(l_{ij})$ with entries indexed by a finite set $C$, we define the element $\left[\Lambda\right]$ in $\mathcal{A}(C)$ by

\centerline{\includegraphics[scale=0.85]{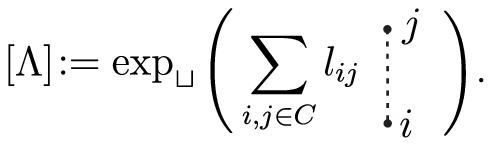}}

It was proved in \cite[Lemma 4.12]{MR2403806} that the LMO functor takes group-like values, and that if $w$ and $u$ are non-associative words in  $\bullet$ of lengths $g$ and $f$ respectively, then for $M\in\mathcal{LC}ob_q(w,u)$, $\widetilde{Z}(M)$ splits as $\widetilde{Z}(M)=\widetilde{Z}^s(M)\sqcup\widetilde{Z}^Y(M)$, where $\widetilde{Z}^Y(M)$ belongs to $\mathcal{A}^Y(\left\lfloor g\right\rceil^+\sqcup\left\lfloor f\right\rceil^{-})$ and $\widetilde{Z}^s(M)$ only contains struts. Moreover $\widetilde{Z}^s(M)$  is given by
\begin{equation}\label{struteqn}
 \widetilde{Z}^s(M)=\left[\frac{\text{Lk}(M)}{2}\right], 
\end{equation}
where $\text{Lk}(M)$ has been defined in subsection \ref{subsection2.2}, see for instance Example \ref{examplepsi}. The colors $1^+,\ldots,g^+$ and $1^-,\ldots,g^-$ in the series of Jacobi diagrams $\widetilde{Z}(M)$ refer to the curves $m_+(\beta_1)$,$\ldots$, $m_+(\beta_g)$ and $m_-(\alpha_1),\ldots, m_-(\alpha_g)$ on the top and bottom surfaces of $M$ respectively.
\subsection{Diagrammatic version of the Johnson-Levine homomorphisms}\label{subsection5.2}

In subsection \ref{subsection2.5} we recalled the diagrammatic version of the Johnson homomorphisms.  The same idea applies to the Johnson-Levine homomorphisms. We have seen in Section \ref{seccion3} that the $k$-th Johnson-Levine homomorphism takes values in $D_k(H')$. Let us consider the space $\mathcal{A}_k^{Y,t,c}(H')=\mathcal{T}_k(H')\otimes\mathbb{Q}$ of connected tree-like Jacobi diagrams of i-deg $=k$ with $H'$-colored legs. We have the isomorphism 
\begin{equation}\label{ecuacion4.15}
\eta_k:\mathcal{A}_k^{Y,t,c}(H')\longrightarrow D_k(H')\otimes\mathbb{Q},\ \ \ \ T\longmapsto\sum \text{color}(v)\otimes (T \text{ rooted at } v),
\end{equation}
 as in subsection \ref{subsection2.5}. Define the \emph{diagrammatic version} of the $k$-th Johnson-Levine homomorphism by
$$\eta_k^{-1}(\tau_k^L(M))\in \mathcal{A}_k^{Y,t,c}(H').$$

\medskip

Moreover, if we consider the symplectic basis $\{a_1,\ldots,a_g,b_1\ldots,b_g\}$ of $H$ fixed in subsection \ref{seccion3.1}, we have that 
$$\eta_k^{-1}(\tau_k^L(M))_{|\iota_*(b_j)\mapsto j^+}\in \mathcal{A}_k^{Y,t,c}(\left\lfloor g\right\rceil^+),$$

\noindent where the expression $\iota_*(b_j)\mapsto j^+$ means to replace the color $\iota_*(b_j)$ by the color $j^+$ for $j=1,\ldots,g$.  We still denote the diagrammatic version by $\tau_k^L$.
\subsection{Relating the LMO functor and the Johnson-Levine homomorphisms}\label{subsection5.4}
 
We have seen that for $M\in J_k^L\mathcal{C}$, the homomorphism $\tau_k^L(M)$ can be seen as  taking  values in the space  $\mathcal{A}_k^{Y,t,c}(\left\lfloor g\right\rceil^+)$ of connected tree-like Jacobi diagrams  with $\left\lfloor g\right\rceil^+$-colored legs of i-deg $=k$. While  the value $\widetilde{Z}(M)$ of the LMO functor takes values in ${}^{ts}\!\!\mathcal{A}(g,g)$. Now,  $\mathcal{A}_k^{Y,t,c}(\left\lfloor g\right\rceil^+)$ is contained in ${}^{ts}\!\!\mathcal{A}(g,g)$. In this subsection we show an explicit relation between the Johnson-Levine homomorphisms and the LMO functor. Let us first start by the strut part of the LMO functor. Consider  the monoid homomorphism
\begin{equation}\label{ecuacion5.101}
\vartheta:\mathcal{LC}\longrightarrow \text{Hom}(A,A), \ \ \ \ M\longmapsto \rho_1(M)|_A.
\end{equation}

Notice that $\mathcal{ILC}=\text{ker}(\vartheta)$. We have the following.

\begin{proposition} For $(M,m)\in\mathcal{LC}$, the homomorphism $\vartheta(M)$ is essentially the strut part $\widetilde{Z}^s(M)$ of the LMO functor not considering struts whose both ends are colored by $\left\lfloor g\right\rceil^-$.
\end{proposition}
\begin{proof}
Consider the same bases of $H$, $A$ and $H_1(M;\mathbb{Z})$ as in the proof of Lemma \ref{lemmalinking}. In these bases the matrix of $\vartheta(M)$ is given by $\Lambda=(\lambda_{ij})$, where $\lambda_{ij}$ are integer coefficients  as in  equation (\ref{eq1lemma}). Besides, $\widetilde{Z}^s(M)=\left[\frac{\text{Lk}(M)}{2}\right]$ and we have seen in the proof of Lemma \ref{lemmalinking} that 
$\text{Lk}(M)=\left( \begin{smallmatrix} 0 & \Lambda^T\\ \Lambda & \Delta \end{smallmatrix} \right)$. In other words, the homomorphism (\ref{ecuacion5.101}) is tantamount to the strut part of the LMO functor not considering struts whose both ends are colored by $\left\lfloor g \right\rceil^-$.
\end{proof}

We now turn to the  trivalent part of the LMO functor. For $M\in\mathcal{LC}$ denote by $\widetilde{Z}^{Y,t,+}(M)$ the element in $\mathcal{A}^{Y,t}(\left\lfloor g\right\rceil^+)$ obtained from $\widetilde{Z}^{Y}(M)$ by sending all terms with loops or with $i^{-}$-colored legs to $0$. 
Let us consider the filtration of $\mathcal{C}$ induced by $\widetilde{Z}^{Y,t,+}$. Specifically, we set 
$$\mathcal{F}_k\mathcal{C}:=\{(M,m)\in \mathcal{ILC}\ |\ \widetilde{Z}^{Y,t,+}(M)=\varnothing + (\text{terms of } \text{i-deg}\geq k) \}.$$

We call $\{\mathcal{F}_k\mathcal{C}\}_{k\geq1}$ the \emph{upper tree filtration} of $\mathcal{C}$.

\begin{proposition}\label{homprop}  Let $M,N\in\mathcal{F}_k\mathcal{C}$ and write $\widetilde{Z}^{Y,t,+}(M)=\varnothing+D_k+ (\text{\emph{i-deg}}>k)$ and $\widetilde{Z}^{Y,t,+}(N)=\varnothing+D'_k+ (\text{\emph{i-deg}}>k)$, where $D_k$ and $D'_k$ are linear combinations of connected Jacobi diagrams in $\mathcal{A}^{Y,t}(\left\lfloor g\right\rceil^+)$ of \emph{i-deg} $=k$. Then 
\begin{equation}\label{ecuacion4.14}
\widetilde{Z}^{Y,t,+}(M\circ N)=\varnothing+ (D_k+D'_k)+(\text{\emph{i-deg}}>k).
\end{equation}
\end{proposition}
\begin{proof}
For simplicity of notation, we write $\text{D}(\cdot)$ instead of $\widetilde{Z}^{Y,t,+}(\cdot)$ and $\hat{\text{D}}(\cdot)$ instead $\widetilde{Z}^Y(\cdot)$.  Suppose that
\begin{equation*}
\text{Lk}(M)=\left( \begin{smallmatrix} 0 & \text{Id}_g\\ \text{Id}_g & \Delta \end{smallmatrix} \right).
\end{equation*}
It follows from Lemma 4.5 in \cite{MR2403806}  that 
\begin{equation}\label{inneraarh}
\hat{\text{D}}(M\circ N)=\Bigg\langle \Big(\hat{\text{D}}(M)_{|j^+\mapsto j^*+j^+ +{\Delta}\cdot j^-}\Big) , \Big(\left[{\Delta}/{2}\right]_{|j^-\mapsto j^*}\Big)\sqcup\Big(\hat{\text{D}}(N)_{|j^-\mapsto j^*+j^-}\Big)\Bigg\rangle_{\left\lfloor g\right\rceil^*},
\end{equation}
where ${\Delta}\cdot j^-=\sum_{p=1}^g l_{pj}p^-$ with ${\Delta}=(l_{pq})$, and for $E,F\in {}^{ts}\!\!\mathcal{A}(\left\lfloor g\right\rceil^*\sqcup C)$, the element $\left\langle E,F\right\rangle_{\left\lfloor g\right\rceil^*}\in {}^{ts}\!\!\mathcal{A}( C)$ is defined as the linear combination of  Jacobi diagrams obtained from $E$ and $F$ by considering all  possible ways of gluing all the $\left\lfloor g\right\rceil^*$-colored legs of $E$ with all the $\left\lfloor g\right\rceil^*$-colored legs of $F$, see \cite{MR1931167,MR1931168} for details about this operation.

\noindent It is possible for $\hat{\text{D}}(M)$ and $\hat{\text{D}}(N)$ to have  diagrams of i-deg $<k$ but with some $\left\lfloor g\right\rceil^-$-colored legs or with loops, thus we need to check that this kind of  diagrams do not contribute any  terms of i-deg $\leq k$ to
 $$(\hat{\text{D}}(M\circ N))_{|j^-,\text{\ loops\ }\mapsto\  0}=\text{D}(M\circ N).$$
 The diagrams with loops remain after the pairing (\ref{inneraarh}), so they do not contribute any term to $\text{D}(M\circ N)$. Let $E$ be a diagram of i-deg $<k$ without loops and having  $\left\lfloor g\right\rceil^-$-colored legs, suppose that $E$ appears in $\hat{\text{D}}(M)$, hence $E':=E_{|j^+\mapsto j^*+j^+ +{\Delta}\cdot j^-}$  still has $\left\lfloor g\right\rceil^-$-colored legs. Therefore all the diagrams obtained from $E'$  after the pairing (\ref{inneraarh}) still have $\left\lfloor g\right\rceil^-$-colored legs, so they do not appear in $\text{D}(M\circ N)$.  Now suppose that  $E$ appears in $\hat{\text{D}}(N)$. In this case $E'':=E_{|j^-\mapsto j^*+j^-}$ can be written as $E''=E_1+E_2$, where $E_1$ is a linear combination of diagrams with $\left\lfloor g\right\rceil^-$-colored legs and $E_2$ is a linear combination of diagrams without $\left\lfloor g\right\rceil^-$-colored legs. The diagrams of $E_1$  do not contribute to $\text{D}(M\circ N)$ as in the previous case. The diagrams obtained from  $E_2$ after the pairing (\ref{inneraarh}) could only contribute diagrams of i-deg $>k$ to $\text{D}(M\circ N)$. Summarizing, we have shown that $\text{D}(M\circ N)=\varnothing + (\text{i-deg}\geq k)$. It remains to show that the terms of i-deg $=k$ are exactly those given by $D_k+D_k'$. This can be easily checked  by using  formula (\ref{inneraarh}).
\end{proof}

The above proposition shows that $\mathcal{F}_k \mathcal{C}$ is a monoid and that we can define homomorphisms 
$$\widetilde{Z}^{Y,t,+}_k:J_k^L\mathcal{C}\longrightarrow\mathcal{A}_k^{Y,t,c}(\left\lfloor g\right\rceil^+),$$
for all $k\geq1$, where  $\widetilde{Z}^{Y,t,+}_k(M)$  denotes the terms of i-deg $=k$ in $\widetilde{Z}^{Y,t,+}(M)$ for  $M\in J_k^L\mathcal{C}$. The following theorem shows that the upper tree filtration coincides with the Johnson-Levine filtration, and makes explicit the relation between the LMO functor and the Johnson-Levine homomorphisms.

\begin{theorem} For all $k\geq 1$,
\begin{equation}\label{ecuacion4.12}
\mathcal{F}_k\mathcal{C}=J_k^L\mathcal{C}.
\end{equation}
Moreover, if $M\in J_k^L\mathcal{C}$ then 
\begin{equation}\label{ecuacion4.13}
\widetilde{Z}^{Y,t,+}(M)=\varnothing + \tau_k^L(M) + (\text{\emph{i-deg}}>k).
\end{equation}
\end{theorem}
\begin{proof}
Notice that if  equality (\ref{ecuacion4.13}) holds then $\mathcal{F}_{k+1}\mathcal{C}=J_{k+1}^L\mathcal{C}$. From the definitions,  $\mathcal{F}_1\mathcal{C}=\mathcal{ILC}=J_1^L\mathcal{C}$. Therefore, it is enough to show that equality (\ref{ecuacion4.12}) implies equality (\ref{ecuacion4.13}) for all $k\geq 1$.

\noindent Let $M\in J_k^L\mathcal{C}=\mathcal{F}_k\mathcal{C}$. Theorem \ref{teorema1} allows us to  write
$$\{M\}_{Y_{k+1}}=\{N\}_{Y_{k+1}}\{P\}_{Y_{k+1}},$$
with $N\in J_k\mathcal{C}$ and $P\in \mathcal{HC}\cap\mathcal{ILC}$. From Lemma \ref{lemma4.10} it follows that  $\tau_k^L(P)=0$. Hence by the invariance of $\tau_k^L$ under $Y_{k+1}$-surgery (see subsection \ref{seccion4.2}), we have
$$\tau_k^L(M)=\tau_k^L(N)+\tau_k^L(P)=\tau_k^L(N).$$
By Proposition \ref{proposition1seccion4}, we conclude that $\tau_k^L(M)$ is equal to the reduction of $\tau_k(N)$ under the map $\iota_*:D_k(H)\rightarrow D_k(H')$.

\noindent Besides, by the invariance under $Y_{k+1}$-surgery of the i-deg $=k$ part of the  LMO functor  and Proposition \ref{homprop}, we have 
$$\widetilde{Z}^{Y,t,+}_k(M)=\widetilde{Z}^{Y,t,+}_k(N)+\widetilde{Z}^{Y,t,+}_k(P).$$

\noindent The monoid $\mathcal{HC}$ is contained in the category of \emph{special Lagrangian cobordisms} introduced in \cite{MR2403806}. For every  cobordism $Q$ of this kind, it was proved in  \cite[Corollary 5.4]{MR2403806} that $(\widetilde{Z}(Q))_{|j^-\mapsto 0}=\varnothing$.  Now $P\in \mathcal{HC}\cap\mathcal{ILC}$, so  we have $\widetilde{Z}^{Y,t,+}_k(P)=0$. It follows from Theorem 8.19 in \cite{MR2403806} that $\widetilde{Z}^{Y,t,+}_k(N)$ is the reduction of $\tau_k(N)$ under the map $\iota_*:D_k(H)\rightarrow D_k(H')$, and so it is for $\widetilde{Z}^{Y,t,+}_k(M)$. Hence the theorem follows.
\end{proof}

As an immediate consequence of the above theorem we have the following corollary.

\begin{corollary} For $M\in J_k^L\mathcal{C}$, the upper tree reduction of the LMO functor of internal degree $k$, $\widetilde{Z}_k^{Y,t,+}(M)$, is independent of the choice of a Drinfeld associator. Moreover,
$$\big(\widetilde{Z}_k^{Y,t,+}(M)\big)_{|j^+\mapsto \iota_*(b_j)}\in \mathcal{T}_k(H')\otimes\mathbb{Q}$$
 is also independent of the choice of the system of meridians and parallels used in the definition of the LMO functor.
\end{corollary}


\bibliographystyle{plain}
\bibliography{Johnson-Levine-hom} 
\end{document}